\documentclass[11pt,a4paper, parskip=half]{scrartcl} 
\usepackage[utf8]{inputenc}
\usepackage[english]{babel}
\usepackage[T1]{fontenc}
\usepackage{lmodern}
\usepackage[left=3cm,right=3cm,top=3cm,bottom=3cm]{geometry}
\usepackage[runin]{abstract}
    \abslabeldelim{.}

\usepackage{todonotes}

\usepackage{amsmath, amsfonts, amssymb, amsthm, amstext}
\usepackage{mathtools}
\usepackage{mathrsfs}
\usepackage{mathdots}
\usepackage{cases}

\usepackage{graphicx}
\usepackage{subcaption}
\usepackage{algorithm}

\font\mfett=cmmib10 at11pt
\font\mfetts=cmmib10 at9pt

\def\gamra{\hbox{\mfett\char013}}
\def\gamras{\hbox{\mfetts\char013}}
\def\berta{\hbox{\mfett\char012}}
\def\balpha{\hbox{\mfett\char011}}
\def\blambda{\hbox{\mfett\char021}}
\def\bslambda{\hbox{\mfetts\char021}}

\usepackage{enumitem}

\usepackage{hyperref} 
\usepackage[capitalise]{cleveref}
	
\newcounter{thm}
\numberwithin{thm}{section}
\numberwithin{equation}{section}

	\newtheoremstyle{myplain}		
			{}			
			{}			
			{\itshape}				
			{}				
			{\sffamily\bfseries}				
			{.}		
			{ }				
			{\thmname{#1}\thmnumber{ #2}\textnormal{\textsf{\thmnote{ (#3)}}}}			
    \newtheoremstyle{mybreak}
            {}{}{}{}{\sffamily\bfseries}{.}{\newline}
            {\thmname{#1}\thmnumber{ #2}\textnormal{\textsf{\thmnote{ (#3)}}}}
	\newtheoremstyle{mydef}
			{}{}{}{}{\sffamily\bfseries}{.}{ }
			{\thmname{#1}\thmnumber{ #2}}
	\newtheoremstyle{myrem}
			{}{}{}{}{\sffamily\itshape}{.}{ }
			{\thmname{#1}\thmnumber{ #2}}

\theoremstyle{myplain}
	\newtheorem{theorem}[thm]{Theorem}
	\newtheorem{lemma}[thm]{Lemma}

\theoremstyle{mybreak}
\theoremstyle{mydef}
	
	\newtheorem{remark}[thm]{Remark}
\theoremstyle{mydef}

		\newcommand{\nn}{\mathbb{N}}

\newcommand{\argmin}{\mathop{\mathrm{argmin}}}

\allowdisplaybreaks

\def\sumprime_#1^#2{
    \setbox0=\hbox{$\scriptstyle{#1}$}
    \setbox1=\hbox{$\scriptstyle{#2}$}
    \setbox2=\hbox{$\displaystyle{\sum}$}
    \setbox4=\hbox{${}^\prime\mathsurround=0pt$}
    \dimen0=.5\wd0 \advance\dimen0 by-.5\wd2
    \ifdim\dimen0>0pt
        \ifdim\dimen0>\wd4 \kern\wd4
        \else\kern\dimen0
        \ifdim\dimen1>\wd4 \kern\wd4
        \else\kern\dimen1
    \fi\fi\fi
\mathop{{\sum}^\prime}_{\kern-\wd4 #1}^{\kern-\wd4 #2}
}

\title{\Large Differential approximation of the Gaussian by short cosine sums with  exponential error decay}
\author{Nadiia Derevianko\footnote{TUM School of CIT, 
Department of Computer Science,
Boltzmannstrasse 3,
85748 Garching b. München,
Germany,  nadiia.derevianko@tum.de}  \footnote{Corresponding author} \quad Gerlind Plonka \footnote{Institute for Numerical and Applied Mathematics, G\"ottingen University, Lotzestr.\ 16-18, 37083 G\"ottingen, Germany, plonka@math.uni-goettingen.de} }
\date{\today}

\begin{document}
	\let\oldproofname=\proofname
	\renewcommand{\proofname}{\itshape\sffamily{\oldproofname}}

\maketitle

\begin{abstract}
In this paper, we propose a method to approximate the Gaussian function on ${\mathbb R}$ by a short cosine sum.  
We generalise and extend the differential approximation method proposed in  \cite{Bell, Scha79}  to approximate $\mathrm{e}^{-t^{2}/2\sigma}$  in the weighted space $L^{2}({\mathbb R},  \mathrm{e}^{-t^{2}/2\rho})$ where $\sigma, \, \rho >0$.  We prove that the optimal frequency parameters $\lambda_1, \ldots , \lambda_{N}$  for this method in the approximation problem $ \min\limits_{\lambda_{1},\ldots, \lambda_{N},  \gamma_{1},  \ldots, \gamma_{N}}\|\mathrm{e}^{-\cdot^{2}/2\sigma} - \sum_{j=1}^{N} \gamma_{j} \, {\mathrm e}^{\lambda_{j} \cdot}\|_{L^{2}({\mathbb R}, \mathrm{e}^{-t^{2}/2\rho})}$,  are zeros of a scaled Hermite polynomial. This observation leads us to a numerically stable approximation method with low computational cost of ${\mathcal O}(N^{3})$ operations.  We derive a direct algorithm to solve this approximation problem based on a matrix pencil method for a special structured matrix. The entries of this matrix are determined by hypergeometric functions. 
 For the weighted $L^{2}$-norm, we prove that the  approximation error decays exponentially with respect to the length $N$ of the sum. An exponentially decaying error in the (unweighted) $L^{2}$-norm is achieved using a truncated cosine sum.
 Our new convergence result for approximation of Gaussian functions by exponential sums of length $N$ shows that exponential error decay rates $e^{-cN}$ are not only achievable for complete monotone functions.

\end{abstract}

\section{Introduction}

Gaussian functions are widely used in statistics, signal processing,  molecular modeling, computational chemistry,  as well as in approximation theory, see e.g. \cite{Greengard2022,HR10,BCS,  DerPr2020}. However, since the Gaussian is not always simple to handle on finite intervals, an exact approximation of the Gaussian is helpful in different contexts. 
In this paper we propose a method to approximate the Gaussian  $f(t)=\mathrm{e}^{-t^{2}/2\sigma}$ for $\sigma>0$ on the real line ${\mathbb R}$ and on symmetric intervals around $0$ by short exponential sums. 
In particular, since $f(t)$ is a symmetric function in $R$, the obtained  approximation 
is a short cosine sum.

In this paper, we mainly study the approximation problem in the weighted space $L^{2}({\mathbb R}, \rho)$ with the norm
$$
 \| f\|^2_{L^{2}({\mathbb R}, \rho)} := \int_{-\infty}^{\infty} |f(t)|^{2} \, {\mathrm e}^{-t^{2}/2\rho} \, {\mathrm d} t 
$$
for some given $\rho >0$.
To find an exponential sum $y(t)= \sum_{j=1}^{N} \gamma_{j} \, {\mathrm e}^{\lambda_{j} t}$ 
with $\gamma_{j}, \, \lambda_{j} \in {\mathbb C}$
that approximates $\mathrm{e}^{-t^{2}/2\sigma}$ on ${\mathbb R}$, we have to solve the minimization problem  
$$ \min_{\bslambda \in {\mathbb C}^{N}, \gamras \in {\mathbb C}^{N}} \|\mathrm{e}^{-\cdot^{2}/2\sigma} - \sum_{j=1}^{N} \gamma_{j} \, {\mathrm e}^{\lambda_{j} \cdot}\|_{L^{2}({\mathbb R}, \rho)}.
$$
This  problem is however non-linear and non-convex and therefore very difficult to solve.
We will use a special approximation method, which is also called differential approximation method.
 Replacing the weighted norm $\| \cdot \|_{L^{2}({\mathbb R}, \rho)}$ by an unweighted norm $\| \cdot \|_{L^{2}}$ we show that our algorithm achieves exponential error decay on any interval $[-L,L]$ for $L>0$  by suitable truncation.

\subsection{Differential approximation method}
\label{sec1.1}
This method, proposed in e.g.\ \cite{Bell, Kam1981,Scha79}, is based on the following observation.
For every exponential sum $y(t)= y_{N}(t)= \sum_{j=1}^{N} \gamma_{j} \, {\mathrm e}^{\lambda_{j} t}$, there exists a differential operator $D_{N}= D_{N}(\blambda)$  given by 
\begin{equation}\label{difoper}
D_{N} f (t):= \frac{d^{N}}{d t^{N}}f(t)+b_{N-1}\frac{d^{N-1}}{d t^{N-1}} f(t)+\ldots+ b_0 f(t), \qquad f  \in C^{N}({\mathbb R}),
\end{equation}
 with constant coefficients $b_{0}, \ldots , b_{N-1} \in {\mathbb C}$ such that 
 $D_{N} y(t) =0. $
 This differential operator is determined by the coefficients of the monomial representation of the  (characteristic) polynomial
\begin{equation}\label{P} P_{N} (\lambda) = \prod_{j=1}^{N} (\lambda- \lambda_{j}) = \lambda^{N} + \sum_{k=0}^{N-1} b_{k} \, \lambda^{k}. 
\end{equation}
Therefore,  assuming that a function $f$ can be well approximated by a short exponential sum, it should be possible to find a differential operator $D_{N}$ such that $D_{N} f$ is ''small''. 
Consequently, we apply the following strategy to approximate $f(t)= {\mathrm e}^{-t^{2}/2\sigma}$ by an exponential sum $y_{N}(t)$.
In a first step, we  determine a polynomial $P_{N}(\lambda)$ of the form (\ref{P}), i.e., we determine the vector ${\mathbf b}=(b_{0}, \ldots , b_{N-1})^{T}$ of coefficients of $P_{N}(\lambda)$ by solving 
\begin{align}\label{minb}
\argmin\limits_{{\mathbf b} \in {\mathbb C}^{N}}\|D_{N} f\|_{L^{2}({\mathbb R}, \rho)}.
\end{align}
Then the zeros $\lambda_{j}$ of the characteristic polynomial $P_{N}(\lambda)$ in (\ref{P}) are taken as the frequencies of the exponential sum to approximate $f$. In a second step, we compute  the vector $\gamra=(\gamma_{1}, \ldots , \gamma_{N})^{T}$ of coefficients of the exponential sum  by solving 
\begin{align}\label{ming}
\argmin\limits_{\gamras \in {\mathbb C}^{N}} \|f - \sum\limits_{j=1}^{N} \gamma_{j} \, {\mathrm e}^{{\lambda_j} \cdot}\|_{L^{2}({\mathbb R}, \rho)}.
\end{align}

\subsection{Contribution of this paper}
As we will show in Subsection \ref{sec:2.1},  the minimization problem 
(\ref{minb}) 
 can be solved analytically for $f(t)= {\mathrm e}^{-t^{2}/2\sigma}$, and the coefficient vector ${\mathbf b}$ can be explicitly given. Moreover, the corresponding  characteristic polynomial $P(\lambda)$ in (\ref{P}) is a normalized scaled Hermite polynomial of degree $N$, such that the zeros $\lambda_{j}$ can be simply precomputed with high accuracy.   This observation has been noticed in \cite{Bell, Scha79} for $\sigma=\rho=1$.  We extended these ideas for $\sigma,\rho>0$.
 
 The second minimization problem (\ref{ming}) leads to an equation system of size $N \times N$, where the coefficient matrix is positive definite.
Since $N$ is small, the obtained algorithm requires only a small computational effort while providing very good approximation results.

While searching for optimal parameter vectors $\blambda, \, \gamra \in {\mathbb C}^{N}$, we  show  that the resulting ordered optimal parameters $\lambda_{j}$  are purely imaginary, i.e., $\lambda_{j} \in {\mathrm i} {\mathbb R}$,  and satisfy   $\lambda_{j}= -\lambda_{{N+1-j}}$, while the ordered optimal parameters $\gamma_{j}$ are real,  satisfying $\gamma_{j}= \gamma_{{N+1-j}}$.
In other words, our algorithm yields a cosine sum of length $\lfloor (N+1)/2\rfloor$.  

In Subsection \ref{appreal} we show that the minimization problem (\ref{minb}) can be rewritten as a matrix pencil problem with special matrices whose entries are defined via hypergeometric functions. This observation leads us to  the interesting side result that the eigenvalues of this special matrix pencil (see (\ref{matpen}) in Section  \ref{appreal}) are zeros of scaled Hermite polynomials. Related ideas can be also found  in  \cite{DI2007}, were eigenvalues of some special matrix pencils are approximated by zeros of orthogonal polynomials.  Furthermore, our  result shows the relation between approximation of the Gaussian in the weighted space $L^{2}({\mathbb R},  \mathrm{e}^{-t^{2}/2\rho})$ and  hypergeometric functions.  This connection was also noticed in  \cite{DerPr2020}, where the Gaussian has been approximated by partial Fourier sums with respect to the spherical Gauss-Laguerre basis.   The matrix pencil approach gives us the opportunity to show that zeros of a scaled Hermite polynomial can be used also for approximation in a finite segment $[-T,T]$ for $T>0$ large enough.
\smallskip

In Section \ref{sec:error}, we show that the proposed method leads to an approximation error 
$$\Big\|\mathrm{e}^{-\cdot^{2}/2\sigma} - \sum\limits_{j=1}^{\lfloor (N+1)/2\rfloor} \tilde{\gamma}_{j} \, \cos( |\lambda_{j}| \cdot) \Big\|_{L^{2}({\mathbb R}, \rho)}
< \textstyle \Big(\frac{r}{\sqrt{2(2r+1)}}\Big)^{N} \, N^{3/4}  
,$$ where $r:=\frac{\rho}{\sigma}$. 
For example, for $r=\frac{1}{2}$, we therefore obtain the error decay rate 
$4^{-N} N^{3/4} < 3^{-N}$,  where $\lfloor (N+1)/2 \rfloor$ is the length of the cosine sum. 
The proof of Theorem \ref{maintheorem} is heavily based on the fact that the Gauss-Hermite quadrature rule leads to exponential  decay rates for Gaussian functions. The proof of Theorem \ref{maintheorem} employs explicitly given suboptimal coefficients $\gamma_{j}$, which are determined by  the weights of the Gauss-Hermite quadrature rule, see formula(\ref{gammaneu}).
Since the Gaussian ${\mathrm e}^{-t^{2}/2\sigma}$ itself decays exponentially, we further derive an error estimate in the $L^{2}({\mathbb R})$ norm of the form 
$$ \textstyle \int\limits_{-\infty}^{\infty}  \Big|  {\mathrm e}^{-t^{2}/2\sigma} - \chi_{[-T,T]} {(t)}\, \sum\limits_{k=1}^{\lfloor {(N+1)/2}\rfloor} \tilde{\gamma}_{k} \, \cos(|{\lambda_{k}|t}) \Big|^{2} \, {\mathrm d} t \le  \frac{\tilde{c}}{16^{N/2}} \, N^{3/2}, $$
using a truncated cosine sum, where the choice of $T$ depends on $N$ and $\sigma$.
 These results  are new and particularly show that not only completely monotone functions can be approximated by short exponential sums of length $N$ with error decay ${\mathrm e}^{-cN}$.

In Section \ref{sec:4}, we compare our approach of Section \ref{infinseg} with other Prony-like methods to approximate  the Gaussian by an exponential sum.
We  study a Prony-like method based on the differential operator, see \cite{PePl2013, StPl2020}, which uses only function and derivative values of the Gaussian at $t_{0}=0$. Our new approach outperforms both of these methods regarding the error, the numerical stability and the computational effort.

Finally, we consider the numerical methods ESPRIT \cite{RK89,PT13} and ESPIRA \cite{DPP21,DPR23} to approximate $f(t) = {\mathrm e}^{-t^{2}/2\sigma}$, which employ a finite number of equidistant function values of $f$. To achieve a cosine sum as a resulting approximation, one needs to employ corresponding variants of ESPRIT and ESPIRA, see e.g.\ \cite{DPR23}. The original algorithms ESPRIT and ESPIRA both provide very good approximation results, but yield complex exponential  sums instead of real cosine sums.

\subsection{Related results}

Exponential sum models are widely used in applied sciences and several algorithms proposed for function approximation by short exponential sums. 
Several authors studied the approximation  of special completely monotone functions, as for example $f(t)= \frac{1}{1+t}$, see \cite{BM05, BM10, BH2005,H19, Kam1976, Kam1979, Kam1981, Scha79}.
Further, the approximation of Bessel functions \cite{BM05,CL20,DPR23}, and of the Dirichlet kernel \cite{BM05,DPP21} has been considered. 
To approximate functions by exponential sums on an interval, often the usual Prony-like reconstruction algorithms can be successfully employed, see e.g.\ \cite{PT13, Pbook}.  Other approaches lead to non-convex minimization problems, which are treated by  iterative methods \cite{OS95, ZP19}. 
As described before, our approach is related to \cite{Kam81, Scha79}, but extends it essentially to derive a stable algorithm and new error estimates for the Gaussian.
 Approximations of the Gaussian by  scaling functions and biorthogonal scaling polynomials can be found in \cite{Lee2009}, while in \cite{DerPr2020}  a spherical Gauss-Laguerre basis has been used.

Unfortunately,  there are not many theoretical results available investigating the  error for approximation by exponential sums  more closely.
Results by Kammler \cite{Kam1981}, Braess, and Hackbusch \cite{Bbook, Br95, BH09}  show that special completely monotone functions on $[0, \infty)$ or on finite intervals $[a,b] \subset [0, \infty)$ can be approximated by exponential sums with exponential error decay  ${\mathrm e}^{-cN}$ or ${\mathrm e}^{-c\sqrt{N}}$  with respect to different (weighted) norms (including $L_{\infty}([a,b])$-norm \cite{Bbook}, Theorem VI, 3.4, weighted $L_{\infty}([a,b])$-norm \cite{Br95}, weighted $L_{1}([0, \infty))$-norm \cite{Br95}, $L^{2}([0, \infty))$-norm  \cite{Kam1981,Br95}, $L^{2}([a, b])$-norm \cite{Kam1981}).
We note that these results also imply that for example the sinc function, as a  product of a completely monotone function and an exponential sum, can be approximated with exponential decay.
Recently,  Koyama \cite{K23} studied exponential sum approximation for finite completely monotone functions on $[0,\infty)$. 
However, for the approximation of the Gaussian on finite intervals or on the real line, we are not aware of any convergence results with exponential decay ${\mathrm e}^{-cN}$. This paper gives such decay errors for the first time. 

 Recently, Jiang and Greengard \cite{Greengard2022} proposed to approximate the Gaussian on $[0, \infty)$ using the inverse Laplace transform
\begin{align}\label{green}  \frac{1}{\sqrt{{4}\pi t}} {\mathrm e}^{-\frac{|x|^{2}}{4t}} = \frac{1}{2\pi {\mathrm i}}
 \int\limits_{\Gamma} {\mathrm e}^{st} \, \frac{1}{2 \sqrt{s}} \, {\mathrm e}^{-\sqrt{s}|x|} \, {\mathrm d}s, 
\end{align} 
 where $\Gamma$ is a suitable contour.
An exponential sum to approximate the Gaussian is then achieved by discretization of this contour integral. 
This idea to derive an approximation by short exponential sums with exponentially decaying error mimics the earlier applied approaches for approximation of complete monotone functions.
Depending on the chosen contour, exponential error rates of $c^{-\sqrt{N}}$ or even $c^{-N}$ can be achieved numerically. Theoretically, this error decay is not  justified so far since the mentioned quadrature rules, see e.g.\ \cite{TW14,T21} do not apply in the considered case.



\section{Differential Method for Approximation of the Gaussian on $\mathbb{R}$}
\label{infinseg}

\subsection{Approximation with frequency parameters being zeros of a scaled Hermite polynomial}
\label{sec:2.1}
We want to approximate $f(t)= {\mathrm e}^{-t^{2}/2\sigma}$ for $\sigma>0$ on ${\mathbb R}$ by an exponential sum of length $N$ using the differential approximation method described in Section \ref{sec1.1}.
In the first step, we  determine  the characteristic polynomial $P_{N}(\lambda)= \lambda^{N} + \sum_{k=0}^{N-1} b_{k} \lambda^{k}$, where ${\mathbf b}$ is the solution of (\ref{minb}). 
%
 Afterwards, the zeros $\lambda_{j}$, $j=1, \ldots , N$, of $P_N$ in (\ref{P})  will serve as the frequencies of the exponential sum to approximate ${\mathrm e}^{-t^{2}/2\sigma}$,  and we solve the least squares problem (\ref{ming}) in a second step.
%

\noindent
\textbf{Step 1.}
 We recall the definition of the  physicist's  Hermite polynomial 
using the Rodrigues formula, see \cite{AS}, 
\begin{equation}\label{rodfor}
H_n(t) := \textstyle (-1)^{n} \mathrm{e}^{t^{2}} \frac{\mathrm{d}^{n}}{\mathrm{d}t^{n}}  \mathrm{e}^{-t^{2}}. 
\end{equation}
Its monomial representation is of the form 
\begin{equation}\label{mono}
H_{n}(t) = \textstyle n! \sum_{\ell=0}^{\lfloor \frac{n}{2} \rfloor} \frac{(-1)^{\ell}}{\ell! (n-2\ell)!} (2t)^{n-2\ell}
\end{equation}
and  $H_{n}$  can be recursively defined with $H_{0}(t) := 1$, $H_{1}(t) := 2t$ and 
\begin{equation}\label{rec} H_{n+1}(t) := {2}t \, H_{n}(t) - H_{n}'(t) = 2 t H_{n}(t) - 2n H_{n-1}(t). 
\end{equation}
Obviously, $H_{n}$ possesses the leading coefficient $2^{n}$.
These Hermite polynomials are orthogonal with respect  to the 
 weight function  $w(t)=\mathrm{e}^{-t^{2}}$, and we have
\begin{equation}\label{orth} \textstyle
\int\limits_{-\infty}^{\infty} H_n(t) \, H_m(t) \, \mathrm{e}^{-t^{2}} \, \mathrm{d} t=\sqrt{\pi} \, 2^{n}\, n! \, \delta_{n,m},
\end{equation}
where $\delta_{m,n}$ denotes the Kronecker symbol. The Rodrigues formula (\ref{rodfor}) implies   for $f(t) = {\mathrm e}^{-t^{2}/2\sigma}$  that 
\begin{align} \textstyle
D_{N} f(t) = \sum\limits_{k=0}^{N} b_{k}  \frac{d^{k}}{d t^{k}} \, {\mathrm e}^{-t^{2}/2\sigma} 
&=  {\mathrm e}^{-t^{2}/2\sigma} \, \sum_{k=0}^{N} b_{k} \, (-1)^{k}(2\sigma)^{-\frac{k}{2}} \,   H_{k} \Big( \frac{t}{\sqrt{2\sigma}} \Big).
\label{2.5}
\end{align}
Let $\rho >0$ be given. To determine the characteristic polynomial $P_N(\lambda)$, we have to compute the vector 
${\mathbf b}=(b_{0},b_{1}, \ldots, b_{N-1})^{T} \in {\mathbb C}^{N}$ that minimizes the functional
\begin{equation}\label{Fbrho0} F_\rho({\mathbf b}):= \| D_{N} {\mathrm e}^{-\cdot^{2}/2\sigma}\|^2_{L^{2}({\mathbb R},\rho)} = \textstyle
\int\limits_{-\infty}^{\infty}  |D_{N} {\mathrm e}^{-t^{2}/2\sigma} |^{2} \, {\mathrm e}^{-t^{2}/2\rho} \, \mathrm{d}t
\end{equation}
with $D_{N}$ in (\ref{difoper}).

 Generalizing the results from \cite{Scha79} for $\sigma=\rho=1$ we provide an explicit presentation of 
 ${\mathbf b}= \argmin\limits_{\tilde{\mathbf b} \in {\mathbb C}^{N}} F_\rho(\tilde{\mathbf b})$.  Moreover, we show that the  characteristic polynomial ${P}_{N}(\lambda) = \lambda^{N} + \sum_{k=0}^{N-1} {b}_{k} \lambda^{k}$ is a scaled Hermite polynomial with $N$ symmetric singular zeros on the imaginary axis. 

\begin{theorem}\label{theo1}
 For $\rho >0$, the minimizing vecor ${\mathbf b} \in {\mathbb C}^{N}$  of 
 the functional $F_\rho({\mathbf b})$ in $(\ref{Fbrho0})$
is given by ${\mathbf b}= ({b}_{0}, {b}_{1}, \ldots , {b}_{N-1})^{T}$ with
\begin{align}
\label{bktilde}
{b}_{k} = \left\{ \begin{array}{ll}  \frac{{N!}}{k! (\frac{N-k}{2})!} \left( \frac{\rho+\sigma}{2\sigma(2\rho+\sigma)} \right)^{(N-k)/2}  & N-k \, \textrm{even},  \\
0 & N-k \, \textrm{odd}. \end{array} \right.
\end{align}
Moreover, the corresponding characteristic polynomial $P_N(\lambda)$ is a weighted scaled Hermite polynomial of degree $N$,
\begin{align} \textstyle 
 {P}_{N}(\lambda) = \lambda^{N} + \sum\limits_{k=0}^{N-1} {b}_{k} \lambda^{k} =\left( -{\mathrm{i}} \, \sqrt{\frac{\rho+\sigma}{2\sigma(2\rho+\sigma)}}  \right)^{N}\,H_N\left({ \mathrm{i}} \, \sqrt{\frac{\sigma(2\rho+\sigma)}{2(\rho+\sigma)}}  \, \lambda \right) . 
 \end{align}
 \end{theorem}

\begin{proof} 
1. From (\ref{2.5}) it follows  with $b_N=1$, $\tau := ( \frac{1}{\sigma} + \frac{1}{2\rho})^{1/2} t $ and $c:= (2+ \frac{\sigma}{\rho})^{-1/2}$ that 
\begin{align} \nonumber
F_\rho({\mathbf b}) &=  \textstyle \int\limits_{-\infty}^{\infty}  |D_{N} ({\mathrm e}^{-t^{2}/2\sigma}) |^{2} \, {\mathrm e}^{-t^{2}/2\rho} \, \mathrm{d}t =  \textstyle \int\limits_{-\infty}^{\infty} \Big| \sum\limits_{k=0}^{N} (-1)^{k} (2\sigma)^{-\frac{k}{2}} b_{k} \, H_{k}\Big( \frac{t}{\sqrt{2\sigma}} \Big) \Big|^{2} \, {\mathrm e}^{-t^{2}( \frac{1}{\sigma} + \frac{1}{2\rho})} \, {\mathrm d}t \\
\label{Fbrho}
&= \textstyle c \sqrt{2\sigma}  \int\limits_{-\infty}^{\infty} | \sum\limits_{k=0}^{N} (-1)^{k} (2\sigma)^{-\frac{k}{2}}\,   b_{k} \, H_{k}(c \tau) |^{2} \, {\mathrm e}^{-\tau^{2}} \, {\mathrm d} \tau .
 \end{align}
 Now, 
 $ q_N(\tau ) := \sum_{k=0}^{N} b_{k} \, (-1)^{k} (2\sigma)^{-\frac{k}{2}}   \, H_{k}(c \tau) $
 is a polynomial  of degree $N$ with leading  coefficient $(-2c)^N (2\sigma)^{-N/2}$. Therefore, it can  be rewritten in the basis of Hermite polynomials
 $$ \textstyle q_N(\tau) =  \Big(\frac{-\sqrt{2}c}{\sqrt{\sigma}}\Big)^N \sum\limits_{\ell=0}^N  \frac{\beta_{\ell}}{2^{N}} H_{\ell}(\tau), $$
 where $\beta_N=1$, since $H_{N}$ has  the leading coefficient $2^{N}$.  With this representation  we obtain from (\ref{Fbrho})
 \begin{align*}
F_\rho({\mathbf b}) &= \textstyle c \sqrt{2\sigma}  \int\limits_{-\infty}^{\infty} | \Big( \frac{- c}{\sqrt{2 \sigma}}\Big)^N  \sum\limits_{\ell=0}^N  \beta_{\ell} H_{\ell}(\tau) |^{2} \, {\mathrm e}^{-\tau^{2}} \, {\mathrm d} \tau  \\
&= \textstyle \Big( \frac{c}{\sqrt{2\sigma}}\Big)^{2N+1}  (2\sigma) \sum\limits_{\ell=0}^N |\beta_{\ell}|^2 \, \int\limits_{-\infty}^{\infty} |H_{\ell}(\tau)|^2 \, {\mathrm e}^{-\tau^{2}} d\tau 
=  \Big(\frac{c}{\sqrt{2\sigma}}\Big)^{2N+1}  (2\sigma) \sum\limits_{\ell=0}^N |\beta_{\ell}|^2  \sqrt{\pi} \, 2^{\ell} \,  \ell!,
\end{align*}
where we have used (\ref{orth}). Therefore, $F_\rho({\mathbf b})$ is minimal if $\beta_{{\ell}}=0$ for $\ell=0, \ldots , N-1$, such that 
\begin{equation}\label{mini}  F_\rho({\mathbf b}) = \min_{\tilde{\mathbf b} \in {\mathbb C}^N} F_\rho(\tilde{\mathbf b}) = c^{2N+1} (2\sigma)^{-N+1/2} \sqrt{\pi} \, 2^{N}\, N! = c^{2N+1} \sigma^{-N+1/2} \sqrt{2\pi} \, N!
\end{equation}
is achieved for $q_N(\tau) = (- \frac{c}{\sqrt{2\sigma}})^N  H_N(\tau)$.  Thus, the definition of $q_N$ implies
$$ \textstyle  H_N(\tau) = \Big( \frac{-\sqrt{2\sigma}}{c}\Big)^{N} \sum\limits_{k=0}^N {b}_k (-1)^k (2\sigma)^{-\frac{k}{2}} \, H_k(c\tau), $$
or equivalently,
\begin{equation}\label{Hscale}
 \textstyle H_N\Big( \frac{\tau}{c} \Big) = \frac{1}{c^N} \sum\limits_{k=0}^N {b}_k (-1)^{N-k} (2\sigma)^{\frac{N-k}{2}} H_k(\tau) 
=  \frac{1}{c^N} \sum\limits_{k=0}^N {b}_{N-k} (-1)^{k} (2\sigma)^{\frac{k}{2}} H_{N-k}(\tau),
\end{equation}
  i.e., the minimizer  ${\mathbf b}=(b_0, \ldots, b_{N-1})^T$ is determined by this expansion.
 
\noindent 
2. We observe that the Hermite polynomials  satisfy the scaling property
\begin{equation}\label{scal} 
\textstyle H_{N}(a \tau) = \sum\limits_{r=0}^{\lfloor N/2 \rfloor} \frac{N!}{r! (N-2r)!}  (a^{2}-1)^{r} \, a^{N-2r} \, H_{N-2r}(\tau), \qquad  N\in \mathbb{N},
\end{equation}
for $a \in {\mathbb R}$,
see e.g. \cite[formula (4.16)]{AW84}.  For $a=c^{-1}$,  comparison with (\ref{Hscale}) yields ${b}_{N-k}=0$ for odd $k$, and 
for $k=2r$,
$$ \textstyle  {b}_{N-2r} = \Big( \frac{1-c^2}{2\sigma} \Big)^r \frac{N!}{r!(N-2r)!} = \Big(\frac{\rho+ \sigma}{2\sigma(2\rho+\sigma)}\Big)^r \frac{N!}{r!(N-2r)!} .
$$

\noindent
3. Finally a comparison  of the characteristic polynomial $P_N(\lambda)$ in (\ref{P}) with ${\mathbf b}$ in (\ref{bktilde}) with (\ref{mono}) implies
\begin{align*}
{P}_{N}(\lambda)&= \textstyle \sum\limits_{r=0}^{\lfloor N/2 \rfloor} \frac{{N!}}{r! (N-2r)! } \left( \frac{\rho+\sigma}{2 \sigma(2\rho+\sigma)} \right)^{r}  \lambda^{N-2r} 
=   \left(-{\mathrm{i}} \, \sqrt{\frac{\rho+\sigma}{2\sigma(2\rho+\sigma)}}  \right)^{N} \,H_N\left(\mathrm{i} \, \sqrt{\frac{\sigma(2\rho+\sigma)}{2(\rho+\sigma)}} \, \lambda \right) .
\end{align*}

 \end{proof}
 
 \begin{remark}
1. Theorem \ref{theo1} shows that the differential approximation method leads to optimal frequency parameters $\lambda_{j}$, which are zeros of a scaled Hermite polynomial and can therefore be precomputed with high accuracy. It remains to solve the minimization problem (\ref{ming}) in the second step.
Note that if  ${\mathbf b}$ is defined as in (\ref{bktilde}), this does not mean that  $F_\rho({\mathbf b})$ decreases as $N$ increases, see (\ref{mini}).\\
2. For $\rho \to \infty$  we obtain the usual norm in $L^{2}({\mathbb R})$.
In this case,  we can derive from Theorem \ref{theo1} that the optimal frequency parameters obtained by this method are the zeros of $H_{N}(\sqrt{\sigma} {\mathrm i} \lambda)$.  \\
 3. The differential approximation method can also be applied to approximate other functions $f$ by short exponential sums. If  $f$ is sufficiently smooth, we can always determine the coefficients $b_j$ to minimize $\|D_N f\|_{L^{2}(\mathbb{R})}$ by applying the Fourier  transform and using the Parseval-Plancherel theorem, see \cite[P.230]{Bell}.
The problem  to determine $D_N(f)$ in (\ref{difoper}) is then equivalent to determining orthogonal polynomials on ${\mathbb R}$ with respect to the weight function $|\widehat{f}|^{2}$.
However, generally this method is very costly, since we first have to construct these orthogonal polynomials (using for example the Gram-Schmidt method) and then to compute their zeros. Only in the special case of Gaussians, we obtain the classical Hermite polynomials.
\end{remark}
 Theorem \ref{theo1} shows that the optimal frequencies $\lambda_{j}$, $j=1, \ldots, N$, are the zeros of the Hermite polynomial
$H_{N}\left( {\mathrm i} \sqrt{\frac{\sigma(2 \rho+\sigma)}{2(\rho+\sigma)}} \lambda \right)$. 
It is a trivial observation that for given zeros $t_j$, $j=1, \ldots , N$ of the Hermite polynomial $H_N(t)$, the scaled polynomial $H_{N}\left( {\mathrm i} \sqrt{\frac{\sigma(2 \rho+\sigma)}{2(\rho+\sigma)}} \lambda \right)$ has the zeros 
\begin{align}\label{zeros}  \textstyle \lambda_j =-{\mathrm i} \sqrt{\frac{2(\rho+\sigma)}{\sigma(2\rho+\sigma)}} t_j, \qquad j=1, \ldots , N.
\end{align}
 In particular, it follows that these zeros are all single zeros on the imaginary axis.
In the following, we always assume that the zeros $t_{j}$ of $H_{N}(t)$ are ordered by size, i.e., 
 $  t_{1} > t_{2} > \ldots  > t_{N}. $
Then, the symmetry of Hermite polynomials implies  that the zeros are symmetric with regard to zero, i.e., $t_{j} = -t_{N+1-j}$ for $j=1, \ldots , N$. Consequently,
we also have $\lambda_{j} = -\lambda_{N+1-j}$, $j=1, \ldots , N$.
As we will show, the obtained exponential sum to approximate the Gaussian ${\mathrm e}^{-t^{2}/2\sigma}$ is therefore a cosine sum.

\noindent
\textbf{Step 2.}  Having determined the frequencies $\lambda_j$, $j=1, \ldots , N$, on the imaginary axis,  we want to determine the vector  ${\gamra}=({\gamma}_{1}, {\gamma}_{2}, \ldots , {\gamma}_{N})^{T} \in {\mathbb C}^{N}$ of optimal coefficients satisfying 
\begin{equation}\label{minnorm}
{\gamra} :=  \textstyle \argmin\limits_{\tilde{\gamras} \in {\mathbb C}^{N}} F(\tilde{\gamra}) := \argmin\limits_{\tilde{\gamras} \in {\mathbb C}^{N}} \Big\| \mathrm{e}^{-\cdot^{2}/2\sigma} - \sum\limits_{j=1}^{N} \tilde{\gamma}_j \mathrm{e}^{\lambda_j \cdot} \Big\|_{L^{2}(\mathbb{R}, \rho )}^{2}\, .
\end{equation}
The minimization problem (\ref{minnorm}) is convex (see, for example, \cite[Chapter 1, §7]{Cbook}) and can be solved as described in \cite[Chapter 4, §1]{Cbook}. We obtain
\begin{align}
F(\gamra) =& \textstyle 
\int\limits_{-\infty}^{\infty} \mathrm{e}^{-t^{2} \left(\frac{1}{\sigma}+\frac{1}{2\rho}  \right)} \,  \mathrm{d} t -\sum_{j=1}^{N}  \gamma_{j} \int\limits_{-\infty}^{\infty} \mathrm{e}^{-t^{2} \left(\frac{1}{2\sigma}+\frac{1}{2\rho} \right) }  \mathrm{e}^{\lambda_j t} \, \mathrm{d} t \notag \\
& - \textstyle   \sum\limits_{m=1}^{N}  \overline{\gamma}_{m} \int\limits_{-\infty}^{\infty}  \mathrm{e}^{-t^{2} \left(\frac{1}{2\sigma}+\frac{1}{2\rho} \right) }  \mathrm{e}^{\overline{\lambda}_m t} \, \mathrm{d} t 
+ \sum\limits_{j=1}^{N} \sum\limits_{m=1}^{N}  \gamma_{j}  \overline{\gamma}_{m} \int\limits_{-\infty}^{\infty} \mathrm{e}^{-t^{2}/2\rho}  \mathrm{e}^{(\lambda_j+\overline{\lambda}_m ) t} \, \mathrm{d} t . \label{apper1}
\end{align}
We use the formula $\int_{-\infty}^{\infty}  \mathrm{e}^{-at^{2}+bt} \, \mathrm{d}t=\sqrt{\dfrac{\pi}{a}} \, \mathrm{e}^{\frac{b^{2}}{4a}}$ for $a>0$ and  introduce the notations
\begin{align}\label{funcG}
g_{j}  &:= \textstyle \int\limits_{-\infty}^{\infty} \mathrm{e}^{-t^{2} \left(\frac{1}{2\sigma}+\frac{1}{2\rho} \right) }  \mathrm{e}^{\lambda_j t} \, \mathrm{d} t =\sqrt{\frac{2\pi \sigma \rho}{\sigma+\rho}} \, \mathrm{e}^{\frac{\sigma \rho {\lambda}_{j}^{2} }{2(\sigma+\rho)}}, \\
\label{funcH}
H_{j,m}  &:= \textstyle \int\limits_{-\infty}^{\infty} \mathrm{e}^{-t^{2}/2\rho}  \, \mathrm{e}^{(\lambda_j+\overline{\lambda}_m ) t} \, \mathrm{d} t = \sqrt{2\pi \rho} \, \mathrm{e}^{(\lambda_{j}+\overline{\lambda}_{m})^2 \rho/2} = 
 \sqrt{2\pi \rho} \, \mathrm{e}^{(\lambda_{j}-{\lambda}_{m})^2 \rho/2}.
\end{align}
  Since all frequencies $\lambda_{j}$, $j=1, \ldots , N$, in (\ref{zeros}) are purely imaginary, it follows that $g_j$ and $H_{j,m}= H_{m,j}$ are real.  Taking into account that $\int_{-\infty}^{\infty} \mathrm{e}^{-t^{2} \left(\frac{1}{\sigma}+\frac{1}{2\rho}  \right)} \,  \mathrm{d} t  = \sqrt{\frac{2\pi \rho \sigma}{2\rho + \sigma}}$,  we get
\begin{align*}
F(\gamra) =&  \textstyle \sqrt{\frac{2\pi \rho \sigma}{2\rho + \sigma}}  - \sum\limits_{j=1}^{N} g_j \, (\gamma_j + \overline{\gamma}_j) + \sum\limits_{j=1}^{N}  \sum\limits_{m=1}^{N} H_{j,m} \, \gamma_j \overline{\gamma}_m.
 \end{align*}
Assuming that $\gamma_{j} = \alpha_{j} + {\mathrm i} \beta_{j}$ for $j=1, \ldots , N$, we obtain 
\begin{align}\label{Ferror} 
F(\gamra) =  \textstyle \sqrt{\frac{2\pi \rho \sigma}{2\rho + \sigma}} - 2 \sum\limits_{j=1}^{N} g_j \, \alpha_{j}
+ \sum\limits_{j=1}^{N}\sum\limits_{m=1}^{N} H_{jm} (\alpha_{j}\alpha_{m} + \beta_{j} \beta_{m}).
\end{align}
For the vector $\gamra$ in (\ref{minnorm}) that minimizes the functional $F$ we obtain the necessary conditions
\begin{align*}
\textstyle \frac{\partial  F(\gamras)}{\partial \alpha_{\ell}} &= \textstyle - 2 g_{\ell} + 2 \sum\limits_{m=1}^{N} H_{\ell m}  \alpha_{m} =0,  \qquad 
\textstyle \frac{\partial  F(\gamras)}{\partial \beta_{\ell}} =  2 \sum\limits_{m=1}^{N} H_{\ell m}  \beta_{m} =0,  \quad \ell=1, \ldots , N.
\end{align*}
In matrix vector representation with ${\mathbf H}_N = (H_{jm})_{j,m=1}^N$, ${\mathbf g} := (g_{j})_{j=1}^{N}$, ${\gamra} = \balpha + {\mathrm i} \berta= (\gamma_{j})_{j=1}^{N}$ it follows that 
$ {\mathbf H}_N \balpha = {\mathbf g}$ and ${\mathbf H}_N \berta ={\mathbf 0}. $
For $N>1$, the coefficient matrix ${\mathbf H}_N$ is real, symmetric and positive definite, since for any vector ${\mathbf x} \in {\mathbb R}^N \setminus \{ {\mathbf 0} \}$ we have 
\begin{align*} {\mathbf x}^T {\mathbf H}_N {\mathbf x} &= \textstyle 
\sqrt{2 \pi \rho} \sum\limits_{j=1}^N \sum\limits_{m=1}^N {x}_j {\mathrm e}^{-|\mathrm{Im} \lambda_j - \mathrm{Im}\lambda_m|^2 \rho/2} x_m  \\
& = \textstyle \sqrt{2 \pi \rho} \sum\limits_{j=1}^N \sum\limits_{m=1}^N {x}_j {\mathrm e}^{-|\lambda_j |^2 \rho/2}  x_m {\mathrm e}^{-|\lambda_m |^2 \rho/2}  \Big( \sum\limits_{\ell=0}^{\infty} \frac{\rho^\ell}{\ell!} (\mathrm{Im}\lambda_j)^\ell (\mathrm{Im} \lambda_m)^\ell \Big)\\
&= \textstyle  \sqrt{2 \pi \rho} \sum\limits_{\ell=0}^{\infty} \frac{\rho^\ell}{\ell!} \Big(  \sum\limits_{j=1}^{N} {x}_j (\mathrm{Im}\lambda_j)^\ell {\mathrm e}^{-|\lambda_j |^2 \rho/2}  \Big)\Big( \sum\limits_{m=1}^N {x}_m (\mathrm{Im} \lambda_m)^\ell {\mathrm e}^{-|\lambda_m |^2 \rho/2}  \Big) >0.
\end{align*}
Thus, $\berta ={\mathbf 0}$, i.e.,  the minimizing vector ${\gamra}$ is real and satisfies 
$ {\mathbf H}_N {\gamra} = {\mathbf g}$.
From $\lambda_{N+1-j} = -\lambda_{j}$ for $j=1, \ldots, N$, it follows that ${\mathbf J}_{N} {\mathbf H}_N {\mathbf J}_{N}={\mathbf H}_N $ and ${\mathbf J}_N {\mathbf g} = {\mathbf g}$, where ${\mathbf J}_{N}= (\delta_{j,N+1-k})_{j,k=1}^{N} $ denotes the counter identity.  Since we have on the one hand 
${\mathbf H}_N {\gamra} = {\mathbf g}$,
and on the other hand
$$  {\mathbf H}_N ({\mathbf J}_{N} \gamra) =  ({\mathbf J}_{N} {\mathbf H}_N {\mathbf J}_{N}) ({\mathbf J}_{N} \gamra) = {\mathbf J}_{N} {\mathbf g} = {\mathbf g}, $$
we can conclude 
that ${\gamra} = {\mathbf J}_{N} {\gamra}$, i.e., $\gamma_{j} = \gamma_{N+1-j}$ for $j=1, \ldots , N$.
To solve the system ${\mathbf H}_N {\gamra} = {\mathbf g}$ in a stable way, we can employ a Cholesky decomposition with pivot.

\begin{algorithm}[ht]
\caption{ Differential Prony-type method for approximation $\mathrm{e}^{-t^{2}/2\sigma}$  in  $L^{2}(\mathbb{R},\rho)$ using precomputed zeros of Hermite polynomials}
\label{alg11}
\small{
\textbf{Input:} parameters $\sigma, \rho>0$, $N \in \nn$ the order of the exponential sum;\\
\phantom{\textbf{Input:}} precomputed zeros $t_{1} > t_{2} > \ldots  > t_{N}$ of the Hermite polynomial $H_N(t)$.

\begin{enumerate}
\item Compute the frequencies $\lambda_j := -{\mathrm i} \sqrt{\frac{2(\rho+\sigma)}{\sigma(2\rho+\sigma)}} t_j$, $j=1,\ldots,N$.

\item Compute coefficients  $\gamma_{j }$, $j=1,\ldots,N$ as the solutions of the system:
$$ \textstyle 
\sum\limits_{j=1}^{N}  \mathrm{e}^{(\lambda_{j} - \lambda_{k})^2 \rho/2} \gamma_j =\sqrt{\frac{ \sigma }{\sigma+\rho}} \, \mathrm{e}^{\frac{\sigma \rho \lambda_{k}^{2} }{2(\sigma+\rho)}}, \ k=1,\ldots,N.
$$
\end{enumerate}
\textbf{Output:}   frequencies $\lambda_j$ with $\lambda_j= -\lambda_{N+1-j}$, coefficients  ${\gamma}_{j}$, with $\gamma_j= \gamma_{N+1-j}$ for $j=1,\ldots,N$ \\
 \phantom{\textbf{Output:}} to approximate $\mathrm{e}^{-t^{2}/2\sigma}$ by  $\sum_{j=1}^{N} \gamma_j \mathrm{e}^{\lambda_j t} $ in  $L^{2}(\mathbb{R},  \rho )$.}
\end{algorithm}

Note that the zeros of the Hermite polynomials can be pre-computed with high accuracy, see \cite{Salzer,Tow}.
The numerical effort of Algorithm \ref{alg11}  is governed by the computational cost to solve the linear system of $N$ linear equations with $N$ unknowns in step 2. This takes at 
 most $\mathcal{O}(N^{3})$ flops and can be further reduced to $\mathcal{O}(N^{2.376})$ (see  \cite{CW}).

For the error of the approximation we obtain from (\ref{Ferror}) (with ${\gamra} = \balpha$)
 and ${\mathbf H}_N \gamra={\mathbf g}$
\begin{align}
F({\gamra}) &= \textstyle  \Big\| \mathrm{e}^{-\cdot^{2}/2\sigma} - \sum\limits_{j=1}^{N} \gamma_j \mathrm{e}^{\lambda_j \cdot} \Big\|^{2}_{L^{2}(\mathbb{R}, \rho )} = \sqrt{\frac{2\pi \rho \sigma}{2\rho + \sigma}} - 2 {\mathbf g}^T {\gamra} + {\gamra}^T {\mathbf H}_N {\gamra}
=
 \textstyle  \sqrt{\frac{2\pi \rho \sigma}{2\rho + \sigma}} - {\mathbf g}^{T} {\mathbf H}_N^{-1} {\mathbf g}. \label{error1}
\end{align}
Application of Algorithm \ref{alg11}  provides an approximation error that decays exponentially if $\rho$ is sufficiently small, see
Figure \ref{fig1} for $\sigma=0.8$ and weights $\rho=1$, $\rho=2$.

\begin{figure}[h]
\centering
\begin{subfigure}{0.46\textwidth}
\includegraphics[width=\textwidth]{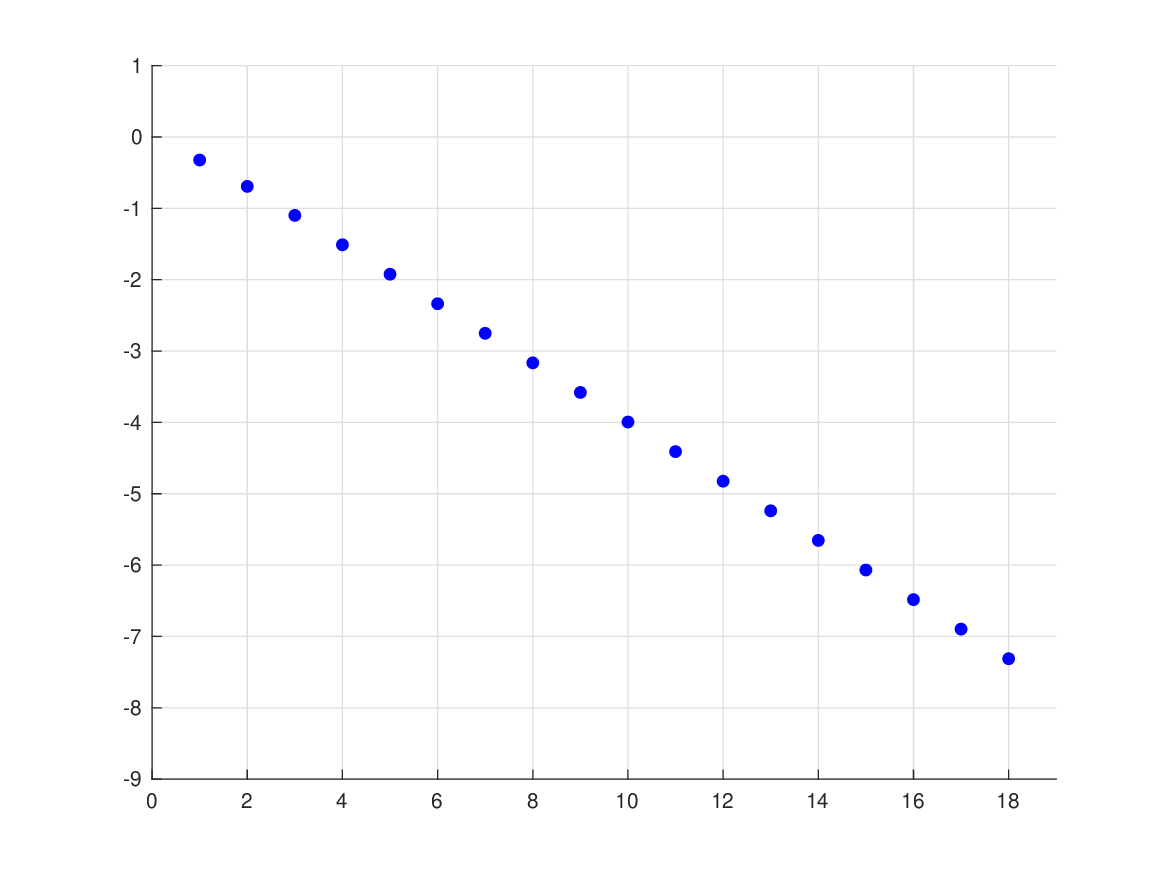}   
\end{subfigure}
\begin{subfigure}{0.46\textwidth}
\includegraphics[width=\textwidth]{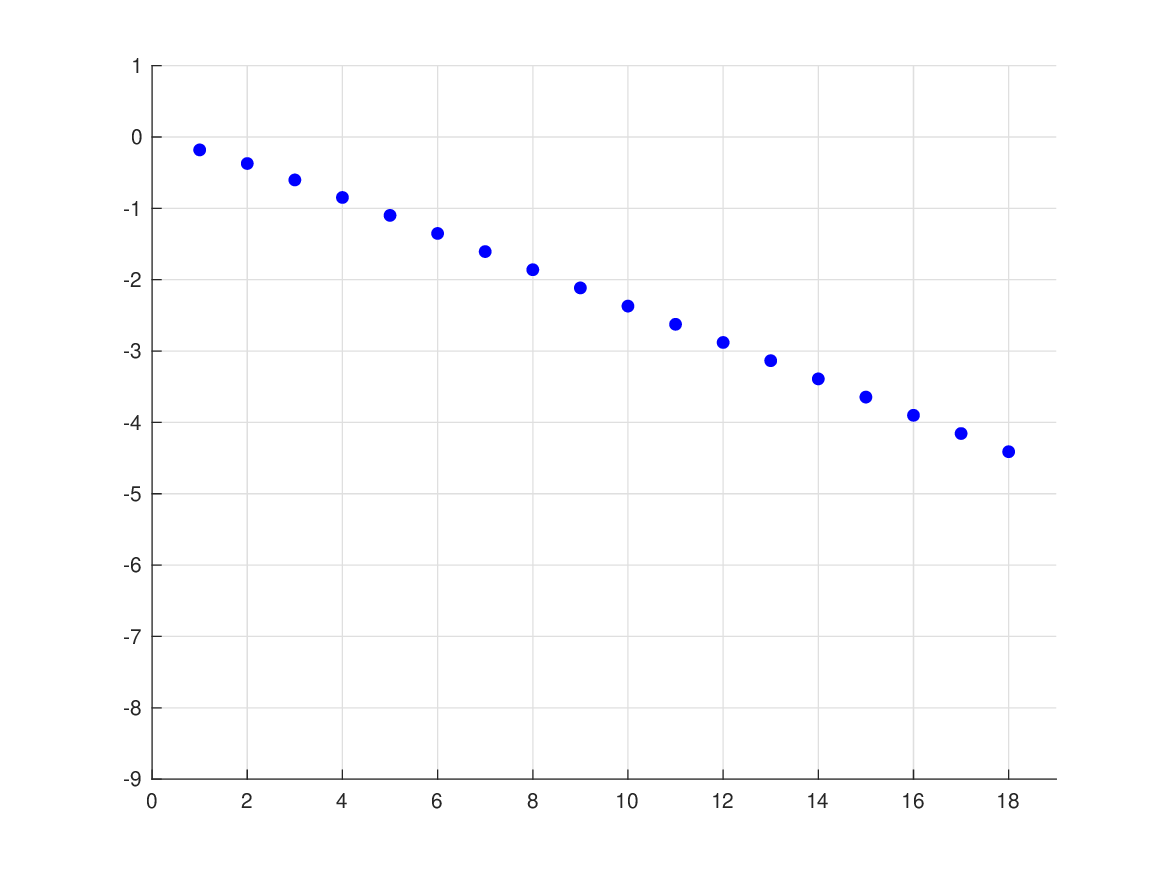}   
\end{subfigure}
\caption{Decay of approximation error in logarithmic scale with respect to $N=1,\ldots,18$, computed with Algorithm \ref{alg11} for $\sigma=0.8$ with $\rho=1$ (left) and $\rho=2$ (right).}
\label{fig1}
\end{figure}

\begin{remark}
1. For the minimization of $F_\rho({\mathbf b})$ in (\ref{Fbrho0}) we can also consider the limit case $\rho \to \infty$, i.e., we can replace the weighted $L^{2}$-norm by the usual $L^{2}$-norm. 
However, for the minimization of  (\ref{minnorm}) it is essential to employ $\rho<\infty$. The reason is obvious. Since the zeros of the Hermite polynomial $H_N$  are symmetric around zero and the coefficients $\gamma_j$ satisfy $\gamma_j= \gamma_{N+1-j}$, we indeed obtain a cosine sum to approximate ${\mathrm e}^{-t^2/2\sigma}$, i.e., 
$$ \textstyle {\mathrm e}^{-t^2/2\sigma} \approx 2 \sum\limits_{j=1}^{ \frac{N}{2}} \gamma_j \, \cos(|\lambda_j|t) \quad \text{or} \quad  {\mathrm e}^{-t^2/2\sigma} \approx \gamma_{(N+1)/2} + 2 \sum_{j=1}^{ \frac{N-1}{2}}  \gamma_j \, \cos(|\lambda_j|t) $$
for even and odd $N$, respectively.
This approximation is only meaningful in the weighted $L^{2}$ norm, i.e., if we multiply  both sides with  the Gaussian window function ${\mathrm e}^{-t^2/2\rho}$.

\noindent
2. Using the symmetry property $\gamma_j = \gamma_{N+1-j}$, the least squares problem in (\ref{minnorm}) can be rewritten. For $N$ even and $\gamra^{(s)}:=(\gamma_1, \ldots , \gamma_{\frac{N}{2}})^T$ we then have to solve
\begin{align*}
{\gamra}^{(s)} = \textstyle  \argmin\limits_{\tilde{\gamras}^{(s)} \in {\mathbb C}^{N}} \Big\|  \mathrm{e}^{-\cdot^{2}/2\sigma} -\sum\limits_{k=1}^{N/2}  \gamma_k \, \cos(\lambda_k \cdot) \Big\|_{L^{2}({\mathbb R}, \rho)}^{2}\, ,
\end{align*}
which leads to a linear system of size $\frac{N}{2}$. For odd $N$, a similar simplification applies.
\noindent
3. Note that  the coefficient vector  $\gamra =(\gamma_j)_{j=1}^N$ in (\ref{minnorm}), determined by ${\mathbf H} {\gamra} = {\mathbf g}$, only depends on the quotient $r := \rho/\sigma$, since  $(2\pi \rho)^{-1/2} {\mathbf H}$ as well as $(2\pi \rho)^{-1/2}{\mathbf g}$ only depend on $r$. 
This can be seen as follows.
Let again $t_1 > t_{2} > \ldots > t_{N}$ be the zeros of $H_N$ and $\lambda_j=-\mathrm{i}\sqrt{\frac{2(\rho+\sigma)}{(2\rho+\sigma)\sigma}} t_j$ for $j=1,\ldots,N$. 
Then, with $\rho=r\sigma$ we obtain for the components of $(2\pi \rho)^{-1/2}{\mathbf H}$ in (\ref{funcH}) that 
\begin{align}
	\textstyle  \mathrm{e}^{(\lambda_{j}-{\lambda}_{k})^2 \frac{\rho}{2}}= 
	  \mathrm{e}^{-{2}\frac{\rho+\sigma}{(2\rho+\sigma)\sigma}(t_j-t_k)^2 \frac{\rho}{2}}  =  \mathrm{e}^{-\frac{(r+1)r}{{2r+1}}(t_j-t_k)^2}   
	\label{Hjk}
\end{align}
and for the components of $(2\pi \rho)^{-1/2}{\mathbf g}$ in (\ref{funcG}),
\begin{align}
\textstyle \frac{1}{\sqrt{2\pi \rho}} g_{j} =	\textstyle \sqrt{\frac{ \sigma }{\sigma+\rho}} \, \mathrm{e}^{\frac{\sigma \rho {\lambda}_{j}^{2} }{2(\sigma+\rho)}} 
	 = \sqrt{\frac{ \sigma }{\sigma+\rho}} \, \mathrm{e}^{-\frac{\rho}{{2\rho+\sigma}} t_j^2} 
	 = \sqrt{\frac{ 1 }{1+r}} \, \mathrm{e}^{-\frac{r}{{2r+1}} t_j^2},
	 \label{gjneu}
\end{align}
such that the coefficients $\gamma_j$ are the solution of the linear system
\begin{align} \label{system1}
	 \textstyle \sum\limits_{j=1}^{N} {\gamma}_j \,  \mathrm{e}^{-\frac{(r+1)r}{2r+1}{(t_j-t_k)^2}}    = \sqrt{\frac{ 1 }{1+r}} \, \mathrm{e}^{-\frac{r}{{2r+1}}t_k^2}, \qquad k=1,\ldots,N.
\end{align}
\end{remark}

\subsection{Differential approximation method as a matrix pencil method}
\label{appreal}

Instead of using Theorem \ref{theo1}, we can solve the minimization problem ${\mathbf b} = \argmin_{\tilde{\mathbf b}} F_{\rho}(\tilde{\mathbf b})$ with the functional $F_{\rho}({\mathbf b}) $ in (\ref{Fbrho0}) directly. From (\ref{Fbrho}) it follows that 
\begin{align*}
F_{\rho}({\mathbf b}) &=  \textstyle c \sqrt{2\sigma}  \sum\limits_{j=0}^{N} \sum\limits_{m=0}^{N} (-1)^{j+m} (2\sigma)^{-(\frac{j+m}{2})} \, b_{j} \overline{b}_{m} \,  \int\limits_{-\infty}^{\infty} H_{j}(c\tau) \, H_{m}(c\tau) \, {\mathrm e}^{-\tau^{2}} d\tau \\
&= \textstyle \sum\limits_{j=0}^{N} \sum\limits_{m=0}^{N} A_{j,m} \, b_{j} \overline{b}_{m},
\end{align*}
where $c= (2+ \frac{\sigma}{\rho})^{{-1/2}} \neq 1$ and 
\begin{equation}\label{coefajk}
 A_{j,m} :=\textstyle c \sqrt{2\rho} (-1)^{j+m} \, (2\sigma)^{-(\frac{j+m}{2})} \int\limits_{-\infty}^{\infty} H_{j}(c\tau) \, H_{m}(c\tau) \, {\mathrm e}^{-\tau^{2}} d\tau. 
\end{equation}
The minimization of $F_{\rho}({\mathbf b})$ then yields the linear system
$$ \textstyle \sum\limits_{m=0}^{N} A_{j,m} \, \tilde{b}_{m} =0, \qquad j=0, \ldots , N-1. $$
or in matrix form, 
\begin{equation}\label{asysdif}
{\mathbf A}_{N,N+1}  \tilde{\mathbf{b}} =0.
\end{equation}
with 
${\mathbf A}_{N, N+1} = (A_{j,m})_{j,m=0}^{N-1,N}$, $\tilde{\mathbf b}= (\tilde{b}_{0}, \ldots , \tilde{b}_{N-1}, 1)^{T}$. Since we are interested in the zeros of the characteristic polynomial $P_{N}(\lambda) = \lambda^{N}+ \sum_{m=0}^{N-1} \tilde{b}_{k} \, \lambda^{m}$ as in (\ref{P}), we apply the matrix pencil method. 
We define the two matrices ${\mathbf A}_{N}(0)=\left(A_{j,m} \right)_{j,m=0}^{N-1}$
and  ${\mathbf A}_{N}(1)=\left(A_{j,m+1} \right)_{j,m=0}^{N-1}$.
Let the companion matrix ${\mathbf C}_N(\tilde{\mathbf{b}})$ of $P_{N}(\lambda)$  be given by 
\begin{equation}\label{commatrix}
{\mathbf C}_N(\tilde{\mathbf b}) =
\left( \begin{array}{ccccc}
0 & 0 & \ldots & 0& -\tilde{b}_0 \\
1 & 0& \ldots  &  0 & -\tilde{b}_1 \\
0 &  1 &  \ldots & 0&  -\tilde{b}_2 \\
\vdots & \vdots &  \vdots &\ddots & \vdots  \\
0 & 0 &  \ldots & 1  & -\tilde{b}_{N-1}
\end{array} \right)
\end{equation}
 with the property 
\begin{equation}\label{propcom}
\mathrm{det}(\lambda {\mathbf I}_N -{\mathbf C}_N(\tilde{\mathbf b}) )=P_N(\lambda).
\end{equation}
 Then (\ref{asysdif}) implies
$
{\mathbf A}_N(0){\mathbf C}_N(\tilde{\mathbf{b}}) ={\mathbf A}_N(1).
$
Taking into account (\ref{propcom}), we find the zeros $\lambda_1, \, \lambda_2,\ldots, \lambda_N$ of $P_N(\lambda)$ by computing the eigenvalues of the matrix pencil 
\begin{equation}\label{matpen}
\lambda {\mathbf A}_N(0)-{\mathbf A}_N(1).
\end{equation}
To improve the numerical stability of this computation we employ the singular value decomposition (SVD) of the matrix   ${\mathbf A}_{N,N+1}$ of the form 
\begin{equation}\label{svda}
{\mathbf A}_{N,N+1}  = {\mathbf U}_{N} \, {\mathbf D}_{N,N+1} {\mathbf W}_{N+1}
\end{equation} 
with  orthogonal matrices ${\mathbf U}_{N} \in {\mathbb R}^{N \times N}$  and  ${\mathbf W}_{N+1} \in {\mathbb R}^{(N+1) \times (N+1)}$.
Then, (\ref{matpen}) can be rewritten as 
\begin{equation}\label{matpen2}
\lambda \,  {\mathbf W}_{N}(0) - {\mathbf W}_{N}(1)
\end{equation}
with the submatrices ${\mathbf W}_{N}(0)={\mathbf W}_{N+1}(1:N,1:N) $  and ${\mathbf W}_{N}(1)={\mathbf W}_{N+1}(1:N,2:N+1)$, where we have used the usual Matlab notation for rows and columns.

\begin{algorithm}[ht]
\caption{Differential Prony-type method for approximation  $\mathrm{e}^{-t^{2}/2\sigma}$ in $L^{2}(\mathbb{R},  \rho )$  using matrix pencil approach}
\label{alg1}
\small{
\textbf{Input:} parameters $\sigma, \rho>0$, $N \in \nn$ the order of an exponential sum;

\begin{enumerate}
\item Create a matrix ${\mathbf A}_{N,N+1}$  with entries in (\ref{coefdef}) and compute the SVD as in (\ref{svda}). 

\item Compute the frequencies  $\lambda_1,\ldots ,\lambda_N$ as eigenvalues of   $\left(  {\mathbf W}_{N}(0)^{T} \right)^{\dagger}  {\mathbf W}_{N}(1)^{T},$
where $\left(  {\mathbf W}_{N}(0)^{T} \right)^{\dagger}$ denotes the Moore-Penrose inverse of  ${\mathbf W}_{N}(0)^{T}$.

\item Compute coefficients  $\gamma_{j }$, $j=1, \ldots ,N$ as the solutions of the system:
$$
\textstyle \sum\limits_{j=1}^{N}  \mathrm{e}^{(\lambda_{j}+\overline{\lambda}_{k})^2 \rho/2}  \gamma_j =\sqrt{\frac{ \sigma }{\sigma+\rho}} \, \mathrm{e}^{\frac{\sigma \rho \overline{\lambda}_{k}^{2} }{2(\sigma+\rho)}}, \qquad k=1,\ldots,N.
$$
\end{enumerate}

\noindent
\textbf{Output:} frequencies $\lambda_j$ 
$j=1,\ldots,N$; 
 coefficients  $\tilde{\gamma}_{j}$, 
 $j=1,\ldots,N$}  to approximate \\
 \phantom{\textbf{Output:}} $\mathrm{e}^{-t^{2}/2\sigma}$ by 
 $\sum_{j=1}^{N} \tilde\gamma_j \mathrm{e}^{\lambda_j t} $
 in the space $L^{2}(\mathbb{R},  \rho )$.
\end{algorithm}

What still remains is the computation of the entries $A_{j,m}$ of the matrix ${\mathbf A}_{N,N+1}$. We  use the following formula \cite[7.374(5)]{GR} 
\begin{align} \nonumber 
 & \textstyle \int\limits_{-\infty}^{\infty} \mathrm{e}^{-2\alpha^{2}t^{2}} H_j(t) H_m(t)\, \mathrm{d} t \\
 & \qquad = \textstyle 2^{\frac{m+j-1}{2}}
\alpha^{-m-j-1} (1-2\alpha^{2})^{\frac{m+j}{2}} \Gamma\left( \frac{m+j+1}{2} \right) \,  _2F_1\left(-m,j,\frac{1-m-j}{2},\frac{\alpha^{2}}{2 \alpha^{2}-1}\right), \label{int1}
\end{align}
 which holds for 
 $j+m$ even, $\alpha \neq \frac{1}{2}$, and 
 $\int_{-\infty}^{\infty} \mathrm{e}^{-2\alpha^{2}t^{2}} H_j(t) H_m(t)\, \mathrm{d} t =0$ for $j+m$  odd.
Here, $\Gamma$ denotes the Gamma function and $_2F_1$ is the hypergeometric function  defined by the series
$$
\textstyle _2F_1(a,b;c,z)= \sum\limits_{n=0}^{\infty} \frac{ (a)_n (b)_n}{(c)_n} \frac{z^{n}}{n!}, \ |z|<1,
$$
where  $(x)_n:=x(x+1) \ldots(x+n-1)$ is the Pochhammer symbol. 
Applying (\ref{int1}) to compute the integral  in (\ref{coefajk}),  we obtain the explicit representation
\begin{equation}\label{coefdef}
A_{j,m}= \textstyle 
(-1)^{\frac{m+j}{2}}  \sqrt{\frac{2 \rho \sigma}{2\rho+\sigma}}  \left( \frac{2(\rho+\sigma)}{\sigma(2\rho+\sigma)}  \right)^{\frac{m+j}{2}} \Gamma\left( \frac{m+j+1}{2} \right)   \,  _2F_1\left(-m,j,\frac{1-m-j}{2},\frac{2\rho+\sigma}{2(\rho+\sigma)} \right) 
\end{equation}
if $j+m$ is even and $A_{j,m}=0$ if $j+m$ is odd.
Interestingly,  these hypergeometric functions $_2F_1$ are also involved in the representation of the approximation error in \cite{DerPr2020}, where the problem of weighted $L^{2}(\mathbb{R}^{3},  \rho )$ approximation of $\mathrm{e}^{-\|t\|^{2}/2\sigma}$ in ${\mathbb R}^3$ by the spherical Gauss-Laguerre basis was considered.
The obtained algorithm is summarized in Algorithm  \ref{alg1}.

\begin{remark}\label{zeros1}
1.  As a corollary of Theorem \ref{theo1} and our observations in this section it follows that the eigenvalues of the matrix pencil (\ref{matpen})  with entries (\ref{coefdef}) are the zeros of  the Hermite polynomials $H_{N}\left( {\mathrm i} \sqrt{\frac{\sigma(2 \rho+\sigma)}{2(\rho+\sigma)}} \lambda \right)$. 
\end{remark}

Since the Gaussian decays exponentially  our method can also be applied for its approximation on a finite interval $[-T,T]$ with properly chosen $T$ (see also Theorem \ref{th37}).  In the last part of this section, we  show that for sufficiently large $T$, the  zeros of $H_{N}\left( {\mathrm i} \sqrt{\frac{\sigma(2 \rho+\sigma)}{2(\rho+\sigma)}} \lambda \right)$ are still suitable frequency parameters. 
To this end we rewrite 
the entries $A_{j,m}$ in (\ref{coefajk}) in a different form.
First observe that for $k \in {\mathbb N}$,  $\alpha >0$,
\begin{align} \nonumber
\textstyle \int\limits_{-\infty}^{\infty} \mathrm{e}^{-\alpha^{2} t^{2}} t^k \, \mathrm{d} t
&= \textstyle  (1+(-1)^k) \,  \int\limits_{0}^{\infty} \mathrm{e}^{-\alpha^{2} t^{2}} t^k \, \mathrm{d} t = \left(\frac{1}{\alpha} \right)^{k+1}  (1+(-1)^k) \,  \int\limits_{0}^{\infty} \mathrm{e}^{- t^{2}} t^k \, \mathrm{d} t   \\
&=\textstyle \frac{1}{2} \, \left(\frac{1}{\alpha} \right)^{k+1}  (1+(-1)^k) \,  \int\limits_{0}^{\infty} \mathrm{e}^{- t} \, t^{\frac{k+1}{2}-1} \, \mathrm{d} t 
\nonumber \\
&
= \textstyle \frac{1}{2} \left( \frac{1}{\alpha}\right)^{k+1} (1+(-1)^{k} )\, \Gamma\left(\frac{k+1}{2} \right), \label{form11}
\end{align}
where the definition  $\Gamma(z)=\int_0^{\infty} t^{z-1} \mathrm{e}^{-t} \, \mathrm{d}t$ of the Gamma function has been applied in the last step.
 Using now the explicit  representation (\ref{mono}) of Hermite  polynomials  in  (\ref{coefajk}),  we compute entries $A_{j,m}$ with (\ref{form11}) as 
\begin{align}\label{coefdef2nd}
A_{j,m}= \textstyle \sqrt{\frac{2 \sigma \rho}{2 \rho+\sigma}}    \left(\frac{2  \rho}{2 \rho+\sigma} \right)^{\frac{j+m}{2}} \sum\limits_{k=0}^{\lfloor j/2 \rfloor}  \sum\limits_{\ell=0}^{\lfloor m/2 \rfloor} \frac{(-1)^{k+\ell}  j! m! }{k! \ell! (j-2k)! (m-2\ell)! }  \left(\frac{2 \rho+\sigma}{4 \rho} \right)^{k+\ell} \Gamma\left( \frac{j+m-2k-2\ell+1}{2} \right)
\end{align}
if $j+m$ is even and $A_{j,m}=0$ if $j+m$ is odd.

\noindent

Let us now compute entries $A_{j,m}(T)$  which are defined as in (\ref{coefajk}) but instead of $(-\infty,\infty)$ we consider the interval $(-T,T)$. Similarly as above, we obtain for $a<0$, $b>0$, $k \in {\mathbb N}$, with the substitution $\eta t^{2}=t$,
\begin{align} \nonumber
& \textstyle   \int\limits_a^{b} \mathrm{e}^{-\eta \, t^{2}} t^{k} \, \mathrm{d}t
= \textstyle  \frac{1}{2\, {\eta}^{(k+1)/2}}  \left( (-1)^{k} \int\limits_{0}^{\eta \, a^{2}}  \mathrm{e}^{-t} t^{(k-1)/2} \, \mathrm{d}t+  \int\limits_0^{\eta \, b^{2}} \mathrm{e}^{-t} t^{(k-1)/2} \, \mathrm{d}t   \right)\\
\nonumber
& = \textstyle \frac{1}{2 \, {\eta}^{(k+1)/2}}  \Big( ( (-1)^{k}+1) \int\limits_{0}^{\infty}  \mathrm{e}^{-t} t^{\frac{k+1}{2}-1} \, \mathrm{d}t - (-1)^{k} \int\limits_{\eta \, a^{2}}^{\infty}  \mathrm{e}^{-t} t^{\frac{k+1}{2}-1} \, \mathrm{d}t -  \int\limits_{\eta \, b^{2}}^{\infty}  \mathrm{e}^{-t} t^{\frac{k+1}{2}-1} \, \mathrm{d}t  \Big)\\
&= \textstyle \frac{1}{2\, {\eta}^{(k+1)/2}}\left( (1+(-1)^{k}) \, \Gamma\left(\frac{k+1}{2} \right)-(-1)^{k} \Gamma\left(\frac{k+1}{2} ,\eta \, a^{2}\right) -\Gamma\left(\frac{k+1}{2} ,\eta \, b^{2}\right)  \right),\label{form1}
\end{align}
where  $\Gamma(z)$ is  the Gamma function and $\Gamma(z,a)=\int_a^{\infty} t^{z-1} \mathrm{e}^{-t} \, \mathrm{d}t$ is  the upper incomplete Gamma function for $\mathrm{Re} \, z>0$. 
Then the explicit representation of Hermite polynomials in (\ref{mono}) yields that $A_{j,m}(T)=0$ if $j+m$ is odd, while for  $j+m$ even we obtain
\begin{align}
A_{j,m}(T)=& \textstyle \sqrt{\frac{2 \sigma \rho}{2 \rho+\sigma}}    \left(\frac{2  \rho}{2 \rho+\sigma} \right)^{\frac{j+m}{2}} \sum\limits_{k=0}^{\lfloor j/2 \rfloor}  \sum\limits_{\ell=0}^{\lfloor m/2 \rfloor} \frac{(-1)^{k+\ell}  j! m! }{k! \ell! (j-2k)! (m-2\ell)! }  \left(\frac{2 \rho+\sigma}{4 \rho} \right)^{k+\ell} \notag \\
 & \textstyle \times 
\Big(  \Gamma\left( \frac{j+m-2k-2\ell+1}{2} \right) - \Gamma\left( \frac{j+m-2k-2\ell+1}{2}, \frac{2 \rho+\sigma}{2 \sigma \rho} T^{2} \right)\Big) . \label{ajmfin1}
\end{align}
The property $\lim\limits_{t \to \infty} \Gamma(z,t)=0$ of the incomplete Gamma function together with (\ref{coefdef2nd}) leads to 
$$
\lim\limits_{T\rightarrow \infty} A_{j,m}(T)=A_{j,m}.
$$
Taking into account Remark \ref{zeros1} we conclude that the zeros of the scaled Hermite polynomial $H_{N}\left( {\mathrm i} \sqrt{\frac{\sigma(2 \rho+\sigma)}{2(\rho+\sigma)}} \lambda \right)$ are also suitable to approximate the Gaussian in $[-T,T]$, if  $T$ is sufficiently large.

\begin{figure}[h]
\centering
\begin{subfigure}{0.46\textwidth}
\includegraphics[width=\textwidth]{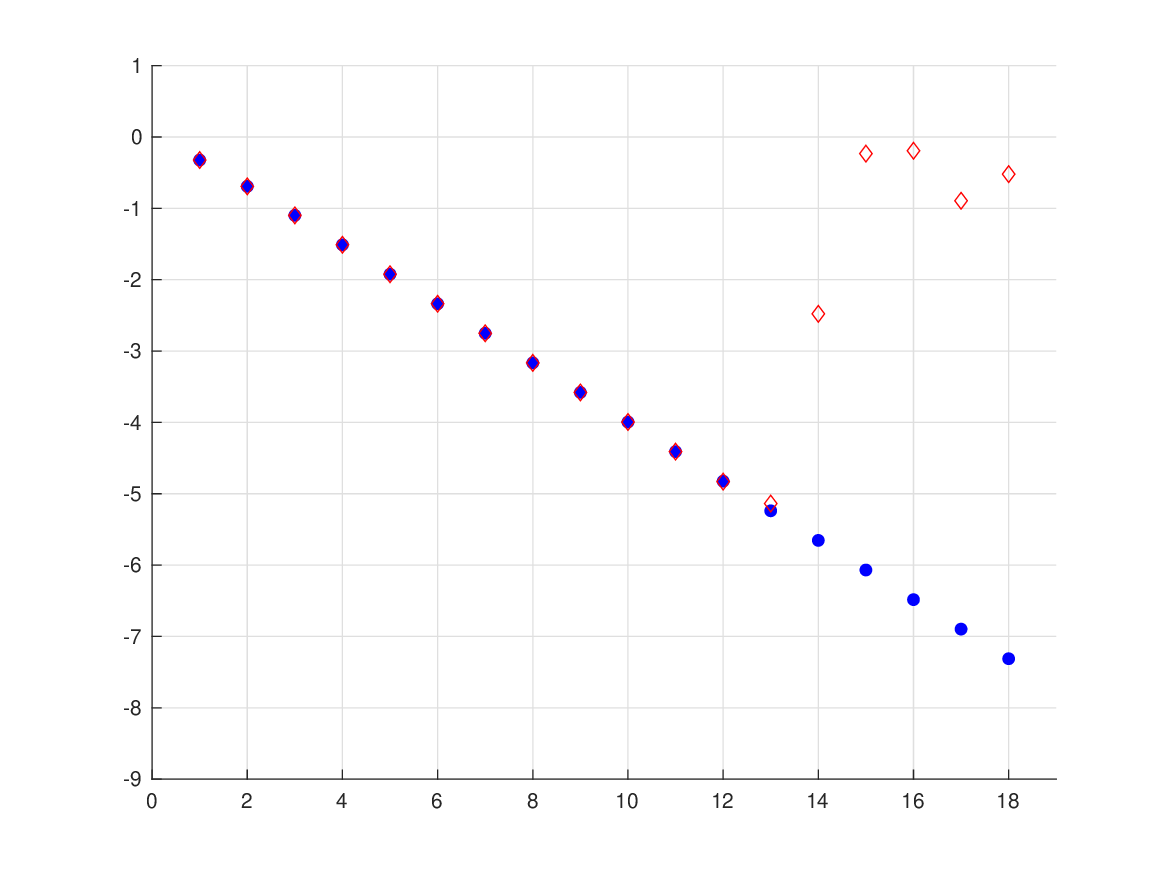}   
\end{subfigure}
\begin{subfigure}{0.46\textwidth}
\includegraphics[width=\textwidth]{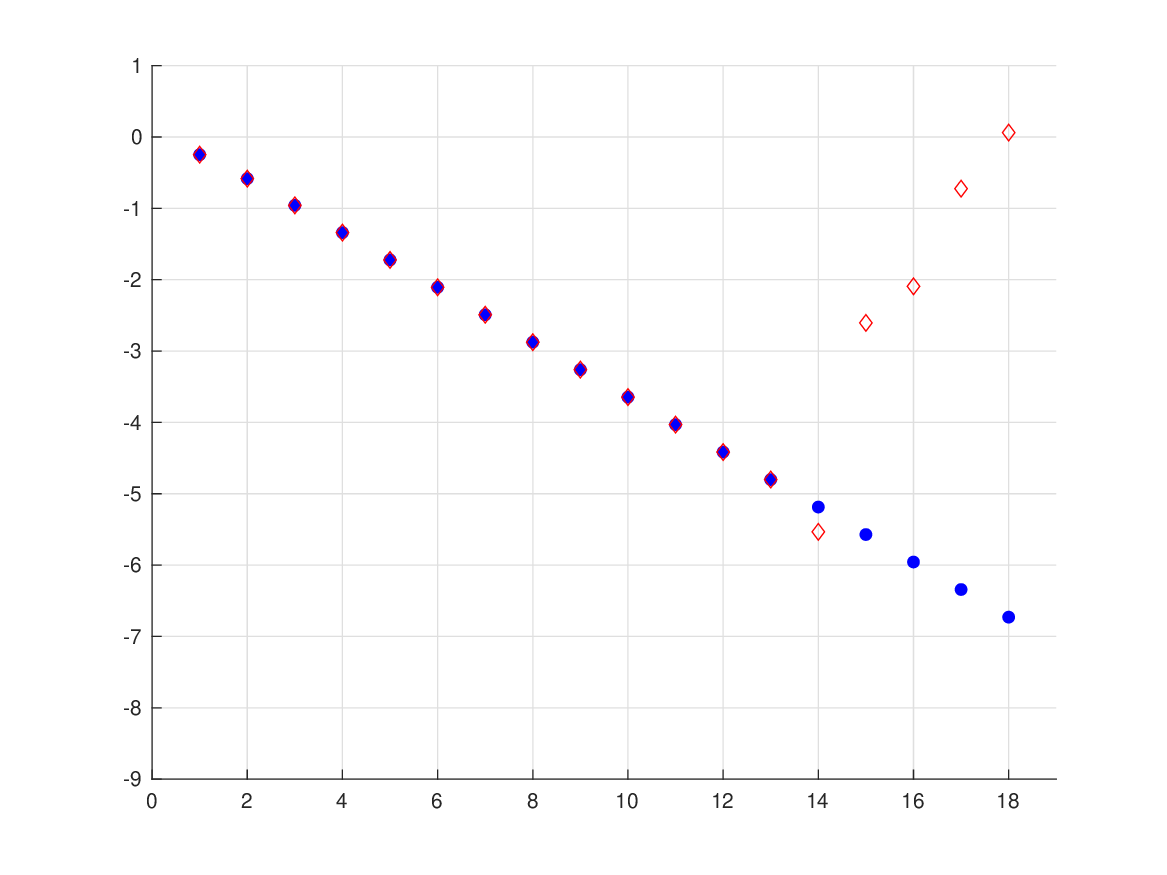}   
\end{subfigure}
\caption{Approximation error in logarithmic scale with respect to $N=1,\ldots,18$ computed with Algorithm \ref{alg11} (blue points) and Algorithm \ref{alg1} (red diamonds) for $\sigma=0.8$ with $\rho=1$ (left) and $\sigma=1.25$ with $\rho=1.75$ (right).}
\label{fig2}
\end{figure}


The overall computational cost of Algorithm \ref{alg1} is $\mathcal{O}(N^{3})$, which is  the complexity of the SVD of an $N\times (N+1)$ matrix computed in the 1st step of Algorithm \ref{alg1} and the numerical complexity to solve the matrix pencil problem in the second step of this algorithm. The linear system  in the 3rd step also takes at most $\mathcal{O}(N^{3})$ operations.

\begin{remark}\label{stab}
1.  Our numerical experiments (in double precision arithmetics) imply that Algorithm \ref{alg1} is less stable for larger $N$ compared to our new Algorithm \ref{alg11}, see Figure \ref{fig2}.
For $N\geq 14$, Algorithm \ref{alg1} requires high precision computations. \\
2. We have developed Algorithm \ref{alg11} particularly for approximation of Gaussian functions. It remains an open question whether  one can achieve approximations of other smooth functions with exponentially decaying error be replacing the frequency parameters resulting from eigenvalues of the matrix pencil by zeros of scaled orthogonal polynomials also in other cases.

\end{remark}



\section{Error estimates for approximation of the Gauss function}
\label{sec:error}

\subsection{Error estimate in the weighted $L^{2}$-norm}
In this section we will show that the approximation error $F({\gamra})$  in (\ref{minnorm}) and (\ref{error1}) decays exponentially with $N$ if $r = \rho/\sigma$ is chosen in a suitable range.
For the proof, we will use an explicit coefficient vector $\gamra$. 
Furthermore, we will employ the Gauss-Hermite quadrature formula, which possesses an exponential rate of convergence for special functions.

Our main result shows  exponential convergence of the approximation of the Gaussian by exponential sums.
\begin{theorem}\label{maintheorem}
Let the zeros of the Hermite polynomial $H_{N}(t)$ in $(\ref{mono})$ be denoted by $t_{1} > t_{2} > \ldots > t_{N}$.
For $\rho>0$ and $\sigma>0$ let $r := \frac{\rho}{\sigma}$ and   $\lambda_j = - {\mathrm i} \sqrt{\frac{2(\rho+\sigma)}{\sigma(2\rho + \sigma)}} t_j$, $j=1, \ldots , N$.
Moreover, let  ${\gamma}_{j}$ be  the coefficients obtained by Algorithm $\ref{alg11}$.
Then, the approximation error  $F({\gamra})$ in $(\ref{minnorm})$  is bounded by 
\begin{align*} \textstyle    \Big\| \mathrm{e}^{-\cdot^{2}/2\sigma} - \sum\limits_{j=1}^{N} \gamma_j \mathrm{e}^{\lambda_j \cdot} \Big\|_{L^{2}(\mathbb{R}, \rho )} 
< \textstyle c   \left(\frac{r}{\sqrt{2(2r+1)}}\right)^N \!  N^{3/4}< \textstyle c \,  \left(\frac{ 3^{1/4} \, r}{\sqrt{2(2r+1)}}\right)^N  ,
\end{align*}
where   $\gamma_j= \gamma_{N+1-j}$, $j=1, \ldots , N$.
The constant $c$ in the estimate depends on $\rho$ but  is independent of  $N$. 
Therefore,  we achieve   an exponential decay of the approximation error for  $r = \frac{\rho}{\sigma} < 2+ \sqrt{6}$.
\end{theorem}


To prove Theorem \ref{maintheorem},  we  need the following notations and preliminary lemmas.

\noindent
We introduce the Gauss-Hermite quadrature formula of the form 
\begin{equation}\label{GH} \textstyle \sum\limits_{j=1}^{N} \omega_{j}\, f({t}_{j})  = \int\limits_{-\infty}^{\infty} {\mathrm e}^{-t^{2}}\,  f(t)\,  dt + R_{f}, 
\end{equation}
where $t_{1}> t_{2}> \ldots > t_{N}$ are the $N$ zeros of the  Hermite polynomial $H_{N}(t)$  and where $R_{f}$ denotes the error of the quadrature formula.
The weights $\omega_{j}$  in (\ref{GH}) are taken as 
\begin{equation}\label{GHweights} \textstyle \omega_{j} := \textstyle \frac{2^{N-1} N!\sqrt{\pi}}{N^{2} [H_{N-1}({t}_{j})]^{2}} = \frac{2^{N+1} N!\sqrt{\pi}}{ [H_{N+1}({t}_{j})]^{2}}, \qquad j=1, \ldots , N, 
\end{equation}
see e.g.  \cite{Szego}, formula (15.3.6), where we have used (\ref{rec}), i.e.,
$$H_{N+1}({t}_j) = 2{t}_j \, H_N({t}_j) - H_N'({t}_j)= -H_N'({t}_j)=-2N \, H_{N-1}({t}_j). $$
The Gauss-Hermite quadrature formula is exact for polynomials of degree up to  $2N-1$, and the error $R_{f}$ can be represented as 
\begin{equation}\label{errorGH}
 \textstyle |R_{f}|= \left| \sum\limits_{j=1}^{N} \omega_{j}\, f({t}_{j})  - \int\limits_{-\infty}^{\infty} {\mathrm e}^{-t^{2}}\,  f(t)\,  dt \right| \le \frac{N!}{(2N)!} \frac{\sqrt{\pi}}{2^{N}} \max_{\xi \in {\mathbb R}}|f^{(2N)}(\xi)|, 
\end{equation}
where $ f^{(2N)}= \frac{\mathrm{d}^{2N}}{\mathrm{d} t^{2N}} f$ denotes the $(2N)$-th derivative of $f$, see \cite{Hildebrandt}, formula (8.7.7).
\noindent
 For functions of the form $f_{k}(t) := {\mathrm e}^{-2s_{1}  t^{2}  +2s_{0} {t}_{k} t}$ with $s_0>0$, $s_1>0$, the error  $R_{f_k}$  can be estimated as follows.
\begin{lemma}\label{lemma-Hermite}
For $s_{0} >0$ and $s_{1} >0$ let $f_{k}(t) := {\mathrm e}^{-2s_{1}  t^{2}  +2s_{0} {t}_{k} t}$, where ${t}_{k}$, $k=1, \ldots , N$, denote the zeros of the Hermite polynomial $H_{N}(t)$ in $(\ref{rodfor})$. Then the error of the Gauss-Hermite quadrature formula for $f_{k}$ satisfies 
\begin{equation}\label{Hermite-error}
\textstyle \Big| \sum\limits_{j=1}^{N} \omega_{j} \, f_{k}({t}_{j})  - \int\limits_{-\infty}^{\infty} {\mathrm e}^{-t^{2}}\,  f_{k}(t)\,  dt \Big| <
 \sqrt{\pi}\,  s_{1}^{N} \, {\mathrm e}^{\frac{s_{0}^{2}}{{2}s_{1}} {t}_{k}^{2}}. 
 \end{equation}
\end{lemma}

\begin{proof}
Applying (\ref{rodfor}), the chain rule and the fact that $H_{2N}$ is even, we obtain
\begin{align*}
\textstyle f^{(2N)}_k(t) &
= \mathrm{e}^{\frac{s_{0}^{2} t_{k}^{2}}{2s_{1}}} 
\frac{\mathrm{d}^{2N}}{\mathrm{d} t^{2N}} \mathrm{e}^{-(\sqrt{2s_1}t-\frac{s_0{t}_k}{\sqrt{2s_1}})^2}  = \textstyle \mathrm{e}^{\frac{s_{0}^{2} t_{k}^{2}}{2s_{1}}} \, (2s_{1})^{N} \, \mathrm{e}^{-(\sqrt{2s_1}t-\frac{s_0{t}_k}{\sqrt{2s_1}})^2}H_{2N}(\sqrt{2s_1}t-\frac{s_0{t}_k}{\sqrt{2s_1}})\\
& = \textstyle (2s_1)^N  {\mathrm e}^{-2s_{1} t^{2} + 2 s_{0} {t}_{k} t} H_{2N}(\frac{-2s_1t+s_0{t}_k}{\sqrt{2s_1}}).
\end{align*}
Therefore,
\begin{align*} \textstyle 
\max\limits_{t \in {\mathbb R}} | f^{(2N)}_{k}(t)| = \textstyle  (2s_{1})^{N}  \max\limits_{x \in {\mathbb R}} \Big({\mathrm e}^{-x^{2}+\frac{s_{0}^{2}}{2s_{1}} {t}_{k}^{2}}  | H_{2N}(x)| \Big)
=  \textstyle (2s_{1})^{N}   {\mathrm e}^{\frac{s_{0}^{2}}{2s_{1}} {t}_{k}^{2}}   \max\limits_{x \in {\mathbb R}} \Big({\mathrm e}^{-x^{2}}   |H_{2N}(x) |\Big). 
\end{align*}
Since $({\mathrm e}^{-x^{2}}  \, H_{2N}(x))'=(-2x \, H_{2N}(x) + H_{2N}'(x)) {\mathrm e}^{-x^{2}} = - H_{2N+1}(x) \, {\mathrm e}^{-x^{2}}$, 
 the local extrema of ${\mathrm e}^{-x^{2}}  \, H_{2N}(x)$ occur at the zeros of $H_{2N+1}(x)$.
We show that the global maximum of ${\mathrm e}^{-x^{2}}  \, |H_{2N}(x)|$ is obtained at $x=0$.
For this purpose, we apply a generalization of the Theorem of Sonin, see \cite{Szego}, Theorem 7.31.1 and the corresponding footnote. This theorem says that, if a function $y(x)$ satisfies the differential equation
$$ (k(x) \, y'(x))' + \phi(x) \, y(x) =0, $$
where  $k(x)>0$ and $\phi(x)>0$ are continuously differentiable, then the relative maxima of $|y(x)|$ for $x\ge 0$ form  an increasing or decreasing  sequence according as $k(x) \phi(x)$ is decreasing or increasing.
We simply observe that $y(x) = {\mathrm e}^{-x^{2}}  \, H_{2N}(x)$ satisfies the above differential equation with $k(x) = {\mathrm e}^{x^{2}}$ and $\phi(x) = (4N+2) {\mathrm e}^{x^{2}}$. Since $k(x) \phi(x)$ is increasing for $x \ge 0$, it follows that the sequence of local maxima of ${\mathrm e}^{-x^{2}}  \, |H_{2N}(x)|$ decreases for $x\ge 0$ and  $N \in {\mathbb N}$. 
Taking into account that $y(x)$ is even and the point $x=0$ is one of the extrema of $y(x)$, we conclude that 
$$ \max_{x \in {\mathbb R}} |f^{(2N)}_{k}(t)| = \textstyle (2s_{1})^{N}  \, {\mathrm e}^{\frac{s_{0}^{2}}{2s_{1}} {t}_{k}^{2}} \,   |H_{2N}(0)| = (2s_{1})^{N}  \, {\mathrm e}^{\frac{s_{0}^{2}}{2s_{1}} {t}_{k}^{2}} \,   \frac{(2N)!}{N!}, $$
and  (\ref{errorGH}) finally implies  (\ref{Hermite-error}).
\end{proof}

\noindent
For the next estimate,  which is crucial for the proof of the convergence Theorem \ref{maintheorem}, we  employ the Hermite functions defined for $n \in {\mathbb N}_{0}$  by
\begin{align}\label{hermitef} \psi_{n}(t):=  \textstyle (2^{n} n! \sqrt{\pi})^{-\frac{1}{2}} \, {\mathrm e}^{-\frac{t^{2}}{2}} H_{n}(t) .
\end{align}


\begin{lemma} \label{lem2}
Let ${t}_1 > {t}_2 > \ldots > {t}_N$ be the $N$ zeros of the $N$-th Hermite polynomial $H_N$ and let $\omega_j$ be the Gauss-Hermite weigths in $(\ref{GHweights})$.
Then we have 
$$ \textstyle M_N := \sum\limits_{k=1}^N \omega_k \, {\mathrm e}^{{t}_k^2} < C \, N^{3/2},$$
where the constant $C$ is independent of $N$.
\end{lemma}

\begin{proof}
1. The weights $\omega_k$ in (\ref{GHweights}) can be rewritten with the Hermite functions in (\ref{hermitef}) as
$$ \textstyle \omega_k= 
 \frac{1}{(N+1) \, {\mathrm e}^{{t}_k^2} [\psi_{N+1}({t}_k)]^2} ,$$
such that 
$$  \textstyle M_N= \sum\limits_{k=1}^{N} \omega_{k} \,  {\mathrm e}^{ {t}_{k}^{2}} = \frac{1}{N+1} \sum\limits_{k=1}^{N} \frac{1}{[\psi_{N+1}({t}_k)]^2}. $$
Then the symmetry of $\psi_{N+1}$ implies that 
\begin{align}\label{MN1}
\textstyle M_N= \frac{2}{N+1}  
 \sum\limits_{k=1}^{N/2} \frac{1}{[\psi_{N+1}({t}_k)]^2} \quad \text{or} \quad 
M_N= \frac{1}{N+1} \left( \frac{1}{[\psi_{(N+1)/2}(0)]^2} + 
2 \sum\limits_{k=1}^{(N-1)/2} \frac{1}{[\psi_{N+1}({t}_k)]^2} \right), 
\end{align}
for even and odd $N$, respectively, where  ${t}_{(N+1)/2}=0$ for odd $N$.
We consider the zeros and local extrema of $H_{N+1}$ and $\psi_{N+1}$ on $[0,\infty)$ more closely. Obviously, $H_{N+1}$ and $\psi_{N+1}$ share the same $\lfloor \frac{N+1}{2} \rfloor$ 
zeros $t_{1}^{(N+1)} > t_{2}^{(N+1)} > \ldots > t_{\lfloor \frac{N+1}{2} \rfloor}^{(N+1)}$ in $[0, \infty)$. 
By (\ref{rec}), the Hermite polynomial $H_{N+1}$ possesses local extrema at the zeros 
$t_{1} > t_{2} >\ldots > t_{n}$ of $H_{N}$, where 
$n:=\lfloor \frac{N+1}{2} \rfloor$. We denote the locations of local extrema of $\psi_{N+1}$ by $t_0^*>t_1^* > t_2^* > \ldots > t_{n}^* \ge 0$. Then we have the well-known interlacing property
\begin{align}
t_0^*>t_1^{(N+1)} > t_1 > t_1^* >t_2^{(N+1)}>t_2>t_2^*>t_3^{(N+1)}>\ldots > t_n^{(N+1)}>t_n>t_n^*\ge 0.
\label{inter}
\end{align}
Theorem 7.6.3 in \cite{Szego} yields
$$ |\psi_{N+1}(t_0^*)| >|\psi_{N+1}(t_1^*)| > |\psi_{N+1}(t_2^*)| > \ldots > |\psi_{{N+1}}(t_n^*)|. $$
Furthermore, we  always have $|\psi_{N+1}(t_k^*)| > |\psi_{N+1}({t}_k)|$, since ${\mathrm e}^{-t^{2}/2}$ is positive and  monotonically 
decreasing  for $t > 0$.
For $k=1, \ldots, n$,  $\psi_{N+1}(t_k^*)$ and $\psi_{N+1}({t}_k)$ have always the same sign, and we obtain for the difference of function values  
\begin{align}|\psi_{N+1}(t_k^*)  - \psi_{N+1}({t}_k)| = |\psi_{N+1}(t_k^*)|  - |\psi_{N+1}({t}_k)| \le |{t}_k- t_k^*|
\,  \max_{\xi \in [t_k^*, {t}_k]} | \psi_{N+1}'(t_k)|, 
\label{ta}
\end{align}
 since $\psi_{N+1}'(t_k^*)=0$ and $|\psi_{N+1}'(\xi)|$ is monotonically increasing in $[t_k^*, {t}_k]$.
Now,  $\psi_{N+1}'({t}_k)$ can by $H_{N+1}'({t}_k)= 2(N+1) \, H_N({t}_k)=0$ be rewritten as
$$ \textstyle  \psi_{N+1}'({t}_k) =  \frac{ \left( -{t}_k\, H_{N+1}({t}_k) + H_{N+1}'({t}_k)\right)}{(\sqrt{\pi}\, 2^{N+1}  \, (N+1)!)^{1/2}}   {\mathrm e}^{-{t}_k^{2}/2}  = -{t}_k\, \psi_{N+1}({t}_k). $$
Hence, (\ref{ta}) yields 
$\textstyle \left| \frac{\psi_{N+1}(t_k^*)}{\psi_{N+1}({t}_k)}\right| \le  |{t}_k - t_k^*| \, |{t}_k| +1$.
For even $N$, we conclude from (\ref{MN1})
$$ \textstyle M_N <  \textstyle \frac{2}{N+1} \sum\limits_{k=1}^{N/2}    \left( \frac{({t}_k- t_k^*) {t}_k + 1}{\psi_{N+1}({t}_k^*)}  \right)^2 <
 \frac{2}{(N+1) \, [\psi_{N+1}(t_{N/2}^*)]^2} \sum\limits_{k=1}^{N/2} (({t}_k- t_k^*) {t}_k + 1  )^2,
$$
and for odd $N$ similarly,
\begin{align*}
M_N <  \textstyle  \frac{1}{N+1} \left( \frac{1}{[\psi_{N+1}(0)]^2} + \frac{2}{[\psi_{N+1}(0)]^2}\sum\limits_{k=1}^{(N-1)/2}    \left(({t}_k- t_k^*) {t}_k + 1  \right)^2 \right).
\end{align*}
2.  According to \cite[formula (14)]{Salzer},  we have the relation $\textstyle \sum_{k=1}^{\lfloor N/2 \rfloor} {t}_k^2 = \frac{N(N-1)}{4}$ for the positive zeros of $H_N$.
Using the interlacing property (\ref{inter}) and 
observing that the largest zero of  $H_{N+1}$ is bounded by ${t}^{(N+1)}_{1} < \sqrt{2N+3}$, see \cite{Szego}, it follows that
$$ \textstyle \sum\limits _{k=1}^{\lfloor N/2 \rfloor} (t_k^*)^2 > \sum\limits_{k=2}^{\lfloor N/2 \rfloor } ({t}_k^{(N+1)})^2 = \frac{(N+1)N}{4} - ({t}_1^{{(N+1)}})^2 > \frac{N^2 + N}{4} -(2N+3) = \frac{N^2}{4}-\frac{7N}{4}-3.$$
Hence,
$$  \textstyle \sum\limits_{k=1}^{\lfloor N/2 \rfloor} {t}_k^2 - (t_k^*)^2  < \frac{N(N-1)}{4} - \frac{N^2}{4}+\frac{7N}{4}+3 = \frac{3N}{2} +3. $$
We conclude 
\begin{align*} \textstyle  
 \textstyle \sum\limits_{k=1}^{\lfloor N/2 \rfloor} 
({t}_k - t_k^*)\, {t}_k < \sum\limits_{k=1}^{\lfloor N/2 \rfloor} 
({t}_k - t_k^*)\, ({t}_k + t_k^*) = \sum\limits_{k=1}^{\lfloor N/2\rfloor} {t}_k^2 - (t_k^*)^2   < \frac{3N}{2} +3, 
\end{align*}
 and therefore 
\begin{align*} 
 \textstyle \sum\limits_{k=1}^{\lfloor N/2 \rfloor} (({t}_{k} - t_{k}^*)  {t}_k +1 )^2
 < \Big(\sum\limits_{k=1}^{\lfloor N/2 \rfloor} (({t}_{k} - t_{k}^*)  {t}_k +1 )\Big)^2 < \tilde{C} N^2
\end{align*} 
for each $N>0$ with some suitable constant $\tilde{C}$ being independent of $N$. 
 
\noindent 
Finally, we have to estimate $[\psi_{N+1}(0)]^2$ for odd $N$ and $[\psi_{N+1}(t_{N/2}^*)]^2$ for even $N$. 
 For odd $N$ we obtain from $(H_{N+1}(0))^2= 2^{N+1} \, (N!!)^2$ that 
\begin{align*} 
[\psi_{N+1}(0)]^2 &= \textstyle \frac{2^{N+1} \, (N!!)^2}{\sqrt{\pi}  \, 2^{N+1} \, (N+1)!} 
= \frac{1}{\sqrt{\pi}  \, 2^{N+1}} \binom{N+1}{\frac{N+1}{2}} >
\frac{1}{\sqrt{\pi}  \, 2^{N+1}} \frac{2^{N+1}}{\sqrt{\pi(N+2)/2}} =  \frac{1}{\pi}\sqrt{\frac{2}{ N+2}},
\end{align*}
where we have used that $\frac{2^{N+1}}{\sqrt{\pi(N+2)/2}}< \binom{N+1}{\frac{N+1}{2}} < \frac{2^{N+1}}{ \sqrt{\pi (N+1)/2}}$ by Stirling's formula. 
Therefore,
\begin{align*} 
M_N &\le  \textstyle \frac{1}{(N+1)[\psi_{N+1}(0)]^2  }  ( 1 + 2\, \tilde{C} \, N^2) < \frac{\pi \sqrt{N+2}}{\sqrt{2}(N+1)} (1+2 \, \tilde{C} \, N^2) < C_o  \, N^{3/2}, 
\end{align*}
with some $C_o$ being independent of $N$.
For even $N$, we conclude from Formulas (8.65.2), (8.65.3) in \cite{Szego} (with the normalization weights considered for the Hermite functions) that  
\begin{align*}
[\psi_{N+1}(t_{N/2}^*)]^2 &=  \textstyle c_e  \, \Big(\frac{1}{\sqrt{\pi} \, 2^{N+1}  \, (N+1)!}  \Big) \, \frac{[H_{N+1}'(0)]^2}{2N+3} = c_e  \, \frac{(2N+2)^2 \, [H_N(0)]^2}{\sqrt{\pi} \, 2^{N+1}  \, (N+1)! (2N+3)} \\
&= \textstyle c_e  \, \frac{(2N+2)^2 \, 2^N \, [(N-1)!!]^2}{\sqrt{\pi} \, 2^{N+1}  \, (N+1)! (2N+3)} 
> c_e \, \frac{(N+1)}{\sqrt{\pi}\,  (N+3/2)}  \, \frac{1}{2^N}  \frac{2^N}{\sqrt{\pi \, (N+1)/2}}  > c_e  \frac{1}{\pi} \frac{1}{\sqrt{N+1}}
\end{align*}
for some $c_e$ being independent of $N$,  and therefore,
\begin{align*} 
M_N &\le  \textstyle \frac{1}{ (N+1)\, [\psi_{N+1}(t_{N/2}^*)]^2  }  (2 \, \tilde{C} \, N^2)  <  {C_e} \, N^{3/2},
\end{align*}
where the constant $C_e$ is independent of $N$.
\end{proof}

\begin{remark} 
Observe that we cannot use the Gauss-Hermite formula (\ref{GH}) to prove Lemma \ref{lem2}, since while 
$\sum_{k=1}^N \omega_k \, {\mathrm e}^{{t}_k^2}$ is bounded by $C N^{3/2}$, we obtain with $f(t) = {\mathrm e}^{{t}^2}$ in (\ref{GH}) only
$ \textstyle \sum_{k=1}^N \omega_k \, {\mathrm e}^{{t}_k^2} = \int_{-\infty}^{\infty} {\mathrm e}^{-t^{2}}\, {\mathrm e}^{t^{2}} {\textrm d} t + R_{f}, $
where the right-hand side is not finite.
\end{remark}

\noindent
With these preliminaries we are ready to  prove Theorem \ref{maintheorem}.

\begin{proof} (of Theorem \ref{maintheorem})
1. We introduce the notations $s_{0} := \frac{r(1+r)}{(2r+1)}$ and $s_{1}:=\frac{r^{2}}{2(2r+1)}$ and  let  again ${t}_k$, $k=1, \ldots, N$, be the ordered zeros of $H_{N}$.
As shown in (\ref{error1}),  the approximation error $F(\gamra)$  is of the form 
$  \textstyle F({\gamra})  = \sqrt{\frac{2\pi \rho}{2r + 1}} - 2 {\mathbf g}^T {\gamra} + {\gamra}^T {\mathbf H}_N {\gamra}, 
$
where the components  $g_j$  of ${\mathbf g}$ are by (\ref{gjneu}) given by 
$$g_{j} =  \textstyle \frac{\sqrt{2\pi \rho}}{\sqrt{1+r}} {\mathrm e}^{-\frac{r}{{2r+1}} t_{j}^{2}} = \frac{\sqrt{2\pi \rho}}{\sqrt{1+r}} {\mathrm e}^{-(s_0-2s_1) {t}_{j}^{2}}.$$
Further,  we recall  from (\ref{Hjk}) that the components of the matrix ${\mathbf H}_N$ can be represented as 
$$ H_{j,k} = H_{k,j}= \sqrt{2\pi \rho} \, {\mathrm e}^{-\frac{(r+1)r}{{2r+1}}(t_{j}-t_{k})^{2} }=  \sqrt{2\pi \rho} \,  {\mathrm e}^{-s_{0}({t}_{j}-{t}_{k})^{2}}.$$
We introduce now the coefficient vector $\gamra^{(H)} = (\gamma_j^{(H)})_{j=1}^N$ with 
\begin{align}\label{gammaneu} \textstyle \gamma_{j}^{(H)} :=
\frac{\sqrt{r+1}}{\sqrt{\pi(2r+1)}} \, {\mathrm e}^{(s_{0}-2s_{1}) {t}_{j}^{2}}\, \omega_{j} , \qquad j=1, \ldots , N,
\end{align} 
where $\omega_j$ are the weights of the Gauss-Hermite formula (\ref{GH}) given in (\ref{GHweights}).
Then, observing that $\sum_{j=1}^{N} \omega_{j} =  \int_{-\infty}^{\infty} {\mathrm e}^{-t^{2}} dt = \sqrt{\pi}$,
we find
\begin{align*}
\textstyle  {\mathbf g}^{T} \gamra^{(H)} = \sum\limits_{j=1}^{N} g_{j} \gamma_{j} 
&= \textstyle \frac{\sqrt{2\pi\rho}}{\sqrt{\pi(2r+1)}}  \sum\limits_{j=1}^{N} \omega_{j} =\frac{\sqrt{2\pi\rho}}{\sqrt{2r+1}} .
\end{align*}
Therefore, with  $\gamra^{(H)}$ in (\ref{gammaneu}), we have
 $ \textstyle   F({\gamra}^{(H)})   =
\textstyle  ({\gamra}^{(H)})^T {\mathbf H}_N {\gamra}^{(H)} - \sqrt{\frac{2\pi \rho }{2r + 1}}. $

\noindent
2. Next, we consider  $({\gamra}^{(H)})^{T} {\mathbf H}_N {\gamra}^{(H)}$. For the components of ${\mathbf H}_N {\gamra}^{(H)}$  we have 
\begin{align} 
\textstyle \sum\limits_{j=1}^{{N}} H_{j,k} \, \gamma_{j}^{(H)}= \textstyle \frac{\sqrt{2 \rho (r+1)}}{\sqrt{2r+1}}   {\mathrm e}^{-s_{0} {t}_{k}^{2}}\sum\limits_{j=1}^{N} \omega_{j} \, {\mathrm e}^{-2s_{1} {t}_{j}^{2} +2s_{0}{t}_{j}{t}_{k} }
 = \frac{\sqrt{2 \rho (r+1)}}{\sqrt{2r+1}}   {\mathrm e}^{-s_{0} {t}_{k}^{2}}\sum\limits_{j=1}^{N} \omega_{j} f_k(t_j).
\label{fk1}
\end{align}
 with  $f_{k}(t) := {\mathrm e}^{(-2s_{1}  t^{2}  +2s_{0} {t}_{k} t)}$ for $k=1, \ldots , N$.
Using the formula $\int_{-\infty}^{\infty}  \mathrm{e}^{-at^{2}+bt} \, \mathrm{d}t=\sqrt{\frac{\pi}{a}} \, \mathrm{e}^{\frac{b^{2}}{4a}}$ for $a>0$, we get
$$ \textstyle \int\limits_{-\infty}^{\infty} {\mathrm e}^{-t^{2}}\,  f_{k}(t)\,  dt = \int\limits_{-\infty}^{\infty} {\mathrm e}^{-t^{2}(1+2s_{1}) + t(2s_{0}{t}_{k})}\,   dt
= \sqrt{\frac{\pi}{1+2s_1}} {\mathrm e}^{\frac{4s_{0}^{2}  {t}_{k}^{2}}{4(1+2s_1)}} =  \sqrt{\frac{\pi(2r+1)}{(r+1)^{2}}} {\mathrm e}^{2 s_{1} {t}_{k}^{2}}$$
and therefore,  the Gauss-Hermite quadrature formula (\ref{GH}) yields
$$ \textstyle \sum\limits_{j=1}^{N} \omega_{j}\, f_{k}({t}_{j}) = \sqrt{\frac{\pi(2r+1)}{(r+1)^{2}}} {\mathrm e}^{2 s_{1} {t}_{k}^{2}}+R_{f_{k}},$$
where the error $R_{f_{k}}$ can be estimated as in Lemma \ref{lemma-Hermite}.
Hence,  (\ref{fk1}) implies
\begin{align*}
({\gamra}^{(H)})^{T} {\mathbf H}_N {\gamra}^{(H)}\! &=\!  \textstyle \sum\limits_{k=1}^N \!\gamma_k^{(H)}\!\sum\limits_{j=1}^{{N}}\! H_{j,k} \gamma_{j}^{(H)} \!= \!
\textstyle \sum\limits_{k=1}^{N} \gamma_{k}^{(H)} \!\Big(\frac{\sqrt{2 \pi \rho}}{\sqrt{r+1}} {\mathrm e}^{(-s_{0}+2s_{1}) {t}_{k}^{2}} \!+\! \frac{\sqrt{2 \rho (r+1)}}{\sqrt{2r+1}} {\mathrm e}^{-s_{0} {t}_{k}^{2}} \, R_{f_{k}}\Big)\\
&= \textstyle \frac{\sqrt{2 \pi \rho}}{\sqrt{2r+1}}  + \frac{\sqrt{2\rho} (r+1)}{\sqrt{\pi}(2r+1)} \sum\limits_{k=1}^{N} \omega_{k}\,  {\mathrm e}^{-2s_{1} {t}_{k}^{2}} \,  R_{f_{k}} .
\end{align*}
Hence, we arrive at 
$$ \textstyle F(\gamra^{(H)}) = ({\gamra}^{(H)})^T {\mathbf H}_N {\gamra}^{(H)} - \sqrt{\frac{2\pi \rho }{2r + 1}} = \frac{\sqrt{2\rho} (r+1)}{\sqrt{\pi}(2r+1)} \sum\limits_{k=1}^{N} \omega_{k} \,  {\mathrm e}^{-2s_{1} {t}_{k}^{2}} \,  R_{f_{k}}. $$

\noindent
3. From Lemma \ref{lemma-Hermite} it follows that 
$ |R_{f_{k}} | <   \sqrt{\pi} \, s_{1}^{N} \, {\mathrm e}^{\frac{s_{0}^{2}}{ 2s_{1}} {t}_{k}^{2}}. $
Therefore we obtain with $-2s_{1}+ \frac{s_{0}^{2}}{2s_{1}} =1$ and $\sum_{k=1}^{N} \omega_{k}= \sqrt{\pi}$,

\begin{align} 
\textstyle F(\gamra^{(H)}) &  \le  \textstyle \frac{\sqrt{2\rho} (r+1)}{\sqrt{\pi}(2r+1)} \sum\limits_{k=1}^{N} \omega_{k}\,  {\mathrm e}^{-2s_{1} \, {t}_{k}^{2}} \,  |R_{f_{k}}|  \leq
\label{app}
 \textstyle   \frac{\sqrt{2\rho} \, (r+1)}{ 2r+1} \, s_{1}^{N} \, \sum\limits_{k=1}^{N} \omega_{k} \, {\mathrm e}^{ {t}_{k}^{2}}.
\end{align}

\noindent
Finally, applying Lemma \ref{lem2} 
 we obtain 
$$ F(\gamra^{(H)}) \le  \tilde{c} \, s_1^N \, N^{3/2}, $$
where   $\tilde{c}$ is independent of $N$.
Exponential decay is achieved for $s_{1} < 1$, i.e., $r < 2+\sqrt{6}$. 
Since the exponential decay is obtained for $\gamra^{(H)}$ and $ F({\gamra})  < F(\gamra^{(H)})$, the assertion of  Theorem \ref{maintheorem} follows.
\end{proof}

\subsection{Error estimate in the $L^{2}({\mathbb R})$-norm}

Since the approximating exponential sum  derived in  Section \ref{infinseg} is  a cosine sum, the error 
$$ \textstyle {\mathrm e}^{-t^{2}/2\sigma} - \sum\limits_{k=1}^{N} {{\gamma}_k} \, {\mathrm e}^{\lambda_{k} t } $$
on the real line needs to be considered in a weighted $L^{2}({\mathbb R})$ norm to achieve an exponentially decaying error.
However, since  ${\mathrm e}^{-t^{2}/2\sigma}$  itself has exponential decay for ${|t|} \to \infty$, we can approximate ${\mathrm e}^{-t^{2}/2\sigma}$ also by the truncated exponential sum $\chi_{[-T, T]}(t) \, \sum_{k=1}^{N} {\gamma}_k \, {\mathrm e}^{\lambda_{k} t }$, where $\chi_{[-T,T]}$ denotes the characteristic function of the interval $[-T, T]$ for positive $T$, and again obtain exponential decay of this approximation.

\begin{theorem}\label{th37}
For  $N >1$, let $\sum_{k=1}^{N} {\gamma}_{k} \, {\mathrm e}^{\lambda_{k}t}$
be the exponential sum computed by Algorithm $\ref{alg11}$ for $\sigma>0$ and $\rho = \frac{\sigma}{2}$. Further, let 
$T:= \sqrt{2 \sigma N \, \ln(2)}$. Then 
\begin{align*}
R_{T} := \textstyle \Big\|  {\mathrm e}^{-\cdot^{2}/2\sigma} - \chi_{[-T,T]}(\cdot)  \sum\limits_{k=1}^{N} {\gamma}_{k}  {\mathrm e}^{\lambda_{k} \cdot} \Big\|^{2}_{L^{2}({\mathbb R})}  \le  \textstyle \tilde{c} \, 2^{-2N} N^{3/2}, 
\end{align*}
where the constant $\tilde{c}$ does not depend on $N$.
\end{theorem}

\begin{proof}
 Theorem \ref{maintheorem}  provides for the  setting $\rho = \frac{\sigma}{2}$ and $r=\frac{\rho}{\sigma} = \frac{1}{2}$,
$$ F({\gamra})  = \textstyle \|{\mathrm e}^{{-\cdot^2}/2\sigma} -  \sum\limits_{k=1}^{N} {\gamma}_{k} \, {\mathrm e}^{\lambda_{k} \cdot} \|^{2}_{L^{2}({\mathbb R},\rho)} <  c \, 2^{-4 N} N^{3/2}. $$
For the $L^{2}({\mathbb R})$-error it follows 
\begin{align*}
R_{T} &<  \textstyle F({\gamra}) + \int\limits_{-T}^{T} \Big|{\mathrm e}^{{-t^2}/2\sigma} -  \sum\limits_{k=1}^{N} {\gamma}_{k} \, {\mathrm e}^{\lambda_{k} t} \Big|^{2} \, \Big(1 - {\mathrm e}^{-t^{2}/\sigma} \Big)\, {\mathrm d} t + 2 \int\limits_{T}^{\infty} {\mathrm e}^{-t^{2}/\sigma} {\mathrm d}{t} .
\end{align*}
From \cite{CDS2003}, we  have  that
$ \textstyle 2 \int\limits_{T}^{\infty} {\mathrm e}^{-t^{2}/\sigma} {\mathrm d}{t} = 2 \sqrt{\sigma}  \, \int\limits_{T/\sqrt{\sigma}}^{\infty} {\mathrm e}^{-t^{2}} \, {\mathrm d} t 
< \sqrt{\pi \, \sigma} \, {\mathrm e}^{-T^{2}/\sigma}. $
Furthermore, 
\begin{align*}
I_{T}(\sigma) &:= \textstyle \int\limits_{-T}^{T} \Big|{\mathrm e}^{{-t^2}/2\sigma} -  \sum\limits_{k=1}^{N} {\gamma}_{k} \, {\mathrm e}^{\lambda_{k} t} \Big|^{2} \, \Big(1 - {\mathrm e}^{-t^{2}/\sigma} \Big)\, {\mathrm d} t \\
&\le  \textstyle \max\limits_{t \in [-T,T]} \Big({\mathrm e}^{t^{2}/\sigma} - 1 \Big) \, \int\limits_{-T}^{T} \Big|{\mathrm e}^{{-t^2}/2\sigma} -  \sum\limits_{k=1}^{N} {\gamma}_{k} \, {\mathrm e}^{\lambda_{k} t} \Big|^{2} \,{\mathrm e}^{-t^{2}/\sigma}  {\mathrm d} t  \\
& \le \textstyle \Big({\mathrm e}^{T^{2}/\sigma} - 1 \Big) \,  c \, 2^{-4 N}  N^{3/2}. 
\end{align*}
Since for $T=\sqrt{2 \sigma N \, \ln(2)}$ we have  ${\mathrm e}^{T^{2}/\sigma} = 2^{2N}$,  
 we can write the estimate
\begin{align*}
R_{T} &< F(\tilde{\gamra}) + I_{T}(\sigma) + \sqrt{\pi \, \sigma} \, {\mathrm e}^{-T^{2}/\sigma} \\
&\le \textstyle  c \, 2^{-4N} N^{3/2} + c 2^{-2N}  N^{3/2} + \sqrt{\pi \, \sigma} 2^{-2N} \le \tilde{c} \, 2^{-2N} \, N^{3/2}.  
\end{align*}
\end{proof}

\section{Relation to other approximation algorithms}
\label{sec:4}



 We want to compare Algorithm \ref{alg11} with  other known numerical approaches for reconstruction of exponential sums, which are based on discrete measurements of the function or its derivatives.
 Note that there are no error estimates available for any of theses approaches which are comparable to our results in Section \ref{sec:error}.
 
\subsection{Prony's method based on differential operator}
We compare our approach in Section \ref{infinseg} with  a Prony-type method based on the differential operator,  see e.g.\ \cite{PePl2013},  \cite{StPl2020}.  
In Section \ref{infinseg}, we computed the differential operator of the form (\ref{difoper}) that minimizes $\| D_{N}f\|_{L^{2}({\mathbb R}, \rho)}$.  
 By contrast, we consider now the discrete values $f^{(k)}({t_0})$, $k=0, \ldots , L$ (with $L\ge 2N-1$) for $f(t) = {\mathrm e}^{-t^{2}/2\sigma}$ and  solve the interpolation problem 
\begin{equation}\label{difpr}
 D_N f^{(k)} (t_{0})= f^{(N+k)}(t_{0}) + \sum\limits_{j=0}^{N-1} b_j f^{(j+k)}(t_{0})=0, \qquad  k=0,\ldots ,L-N,
 \end{equation}
 for some suitable $t_{0} \in {\mathbb R}$, to evaluate the coefficient vector ${\mathbf b}=(b_{0}, \ldots , b_{N-1},1)^{T}$ determining $D_{N}$.  The rationale  behind this approach is the following.
If $f$ were an exponential sum, then it would lead to the reconstruction of $f$, since for 
$f= \sum_{\nu=1}^{N} \gamma_{\nu} \, {\mathrm e}^{\lambda_{\nu}t}$ we obtain  for $t\in {\mathbb R}$ and $k=0,1,2,\ldots $
\begin{align}\nonumber
D_N f^{(k)} (t) &= \sum_{\nu=1}^{N} \gamma_{\nu} \, \lambda_{\nu}^{N+k} \, {\mathrm e}^{\lambda_{\nu}t} + \sum\limits_{j=0}^{N-1} b_j  \Big( \sum_{\nu=1}^{N} \gamma_{\nu} \, \lambda_{\nu}^{j+k} \, {\mathrm e}^{\lambda_{\nu}t} \Big) \\
\label{Prony1}
&= \sum_{\nu=1}^{N} \gamma_{\nu}  \lambda_{\nu}^{k} \, \Big( \lambda_{\nu}^{N} + \sum\limits_{j=0}^{N-1} b_j  \lambda_{\nu}^{j} \Big) \, {\mathrm e}^{\lambda_{\nu}t} 
 = \sum_{\nu=1}^{N} \gamma_{\nu}  \lambda_{\nu}^{k} P_N(\lambda_{\nu}) {\mathrm e}^{\lambda_{\nu}t}=0
\end{align}
with $P_{N}(\lambda)$ in (\ref{P}).
The equations in (\ref{difpr}) can be rewritten as 
\begin{equation}\label{prmat}
{\mathbf H}_{L-N,N+1}  \mathbf{b}^{T} = {\mathbf 0},
\end{equation}
with the Hankel matrix 
\begin{equation}\label{hankmat}
{\mathbf H}_{L-N,N+1}=\left( f^{(k+\ell)}(t_{0}) \right)_{k=0,\ell=0}^{L-N-1,N}.
\end{equation}
For $f(t) = {\mathrm e}^{-t^{2}/2\sigma}$ and $t_{0}=0$, the entries of the Hankel matrix are obtained in the form
\begin{align} \label{fk0}
f^{(k)}(0) &= \frac{{\mathrm d}^{k}}{{\mathrm d} t^{k}} {\mathrm e}^{-t^{2}/2\sigma}\left|_{t=0} \right. = (-1)^{k} \, (2\sigma)^{-\frac{k}{2}} \, H_{k}(0) 
=  \left\{ \begin{array}{ll}  (-1)^{{\frac{k}{2}}} \,  \sigma^{-\frac{k}{2}} \, (k-1)!! &  k \, \mathrm{even} \\
 0 &  k \, \mathrm{odd}.
\end{array} \right.
\end{align}
 Similarly as in Section \ref{appreal}  relations (\ref{difpr}) and (\ref{Prony1}) imply
 that the wanted frequencies $\lambda_1, \ldots, \lambda_N$, i.e., the zeros of the characteristic polynomial $P_{N}$ in (\ref{P}) are eigenvalues of matrix pencil 
\begin{equation}\label{matpenhan}
\lambda {\mathbf H}_{L-N, N}(0)-{\mathbf H}_{L-N,N}(1),
\end{equation}
where ${\mathbf H}_{L-N, N}(0) = (f^{(k+\ell)}(0))_{k=0,\ell=0}^{L-N-1,N-1}$ and 
${\mathbf H}_{L-N, N}(1) = (f^{(k+\ell)}(0))
_{k=0,\ell=1}^{L-N-1,N}$ are submatrices of ${\mathbf H}_{L-N, N+1}$. 
To solve this matrix pencil problem numerically  we employ the  SVD of ${\mathbf H}_{L-N,N+1} $ in (\ref{hankmat}) of the form 
${\mathbf H}_{L-N,N+1}  = {\mathbf U}_{L-N} \, {\mathbf D}_{L-N,N+1} {\mathbf W}_{N+1}$, 
with unitary square matrices ${\mathbf U}_{L-N}$ and ${\mathbf W}_{N+1}$.
Then, the frequencies can be also found by solving the matrix pencil problem 
\begin{equation}\label{matpen3}
\lambda \,  {\mathbf W}_{N}(0) - {\mathbf W}_{N}(1)
\end{equation}
with ${\mathbf W}_{N}(0)={\mathbf W}_{N+1}(1:N,1:N) $, ${\mathbf W}_{N}(1)={\mathbf W}_{N+1}(1:N,2:N+1)$. 
Finally, the coefficients $\gamma_j$ are computed by solving the overdetermined linear system
\begin{equation}\label{coefpr}
\sum\limits_{j=1}^{N} \gamma_j \,  \lambda_j^{k} =f^{(k)}(0), \ k=0,\ldots, L,
\end{equation}
where the values $f^{(k)}(0)$ are given by (\ref{fk0}). This method is summarized in Algorithm \ref{alg3}. The numerical complexity of Algorithm {\ref{alg3}} is governed by the SVD of the Hankel matrix of size $L-N \times N+1$ and is again  $\mathcal{O}(N^{3})$ flops for $2N-1 \le L \le c N$ for some constant $c$. 

In Figure \ref{fig21}, we compare Algorithm \ref{alg11} and Algorithm \ref{alg3} with regard to the weigthed  $L^{2}({\mathbb R},\rho)$-norm. 
We observe that Algorithm \ref{alg11} provides a better  approximation error and higher numerical stability for larger $N$.

\begin{figure}[h]
\centering
\begin{subfigure}{0.46\textwidth}
\includegraphics[width=\textwidth]{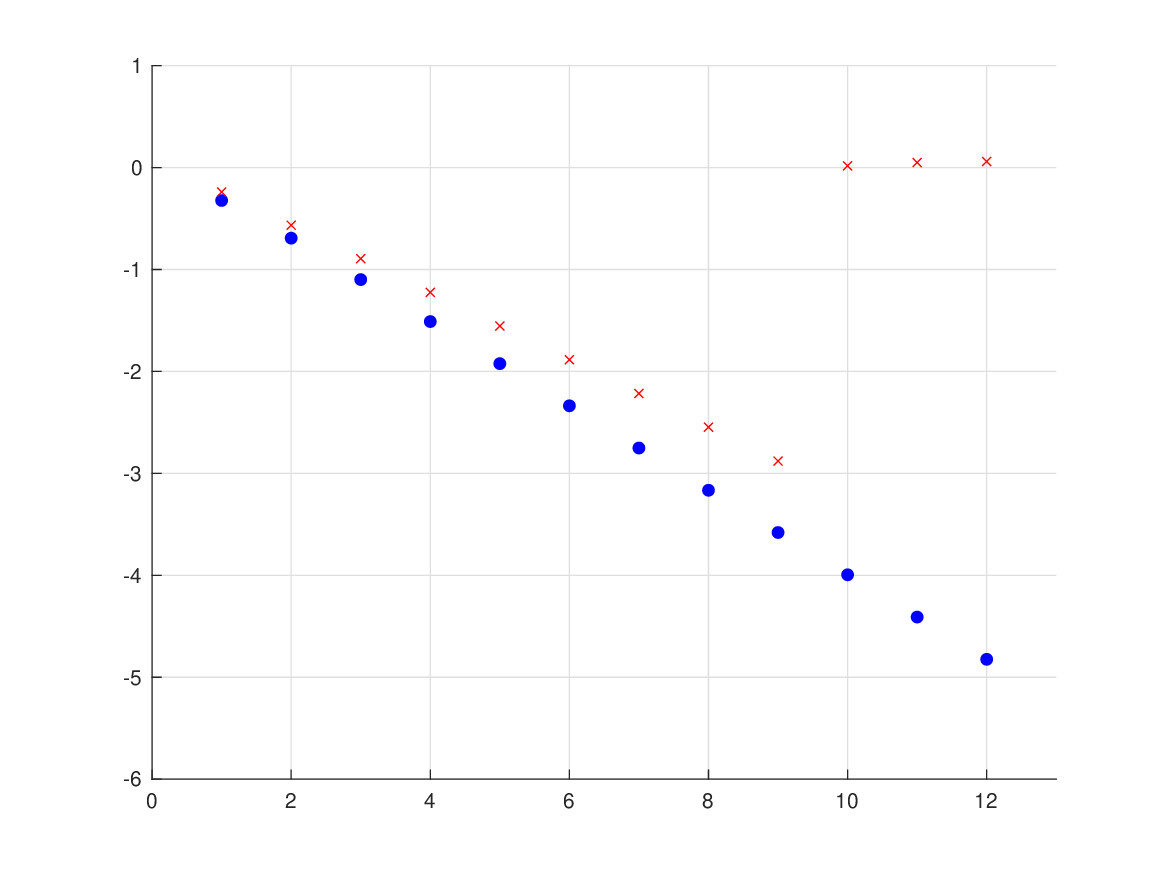}   
\end{subfigure}
\begin{subfigure}{0.46\textwidth}
\includegraphics[width=\textwidth]{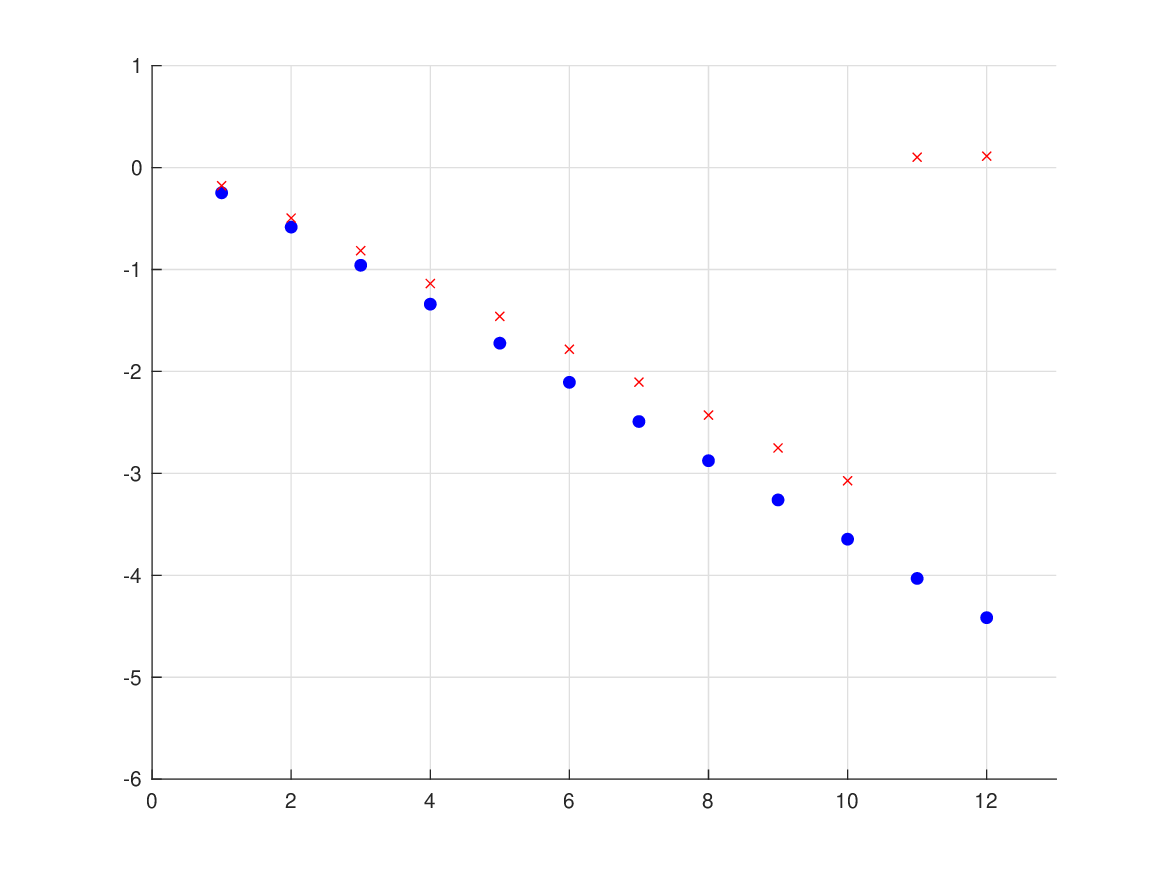}   
\end{subfigure}
\caption{Approximation error in logarithmic scale with respect to $N=1,\ldots,12$ computed with Algorithm \ref{alg11} (blue points) and Algorithm \ref{alg3} (red crosses) for $\sigma=0.8$ with $\rho=1$ (left) and $\sigma=1.25$ with $\rho=1.75$ (right).}
\label{fig21}
\end{figure}

\begin{algorithm}[ht]
\caption{Prony's method based on differential operator for   approximation of $\mathrm{e}^{-t^{2}/2\sigma}$ }
\label{alg3}
\small{
\textbf{Input:} parameters $\sigma>0$, $N \in \nn$ the order of an exponential sum, $t_0=0$, $L \ge 2N-1$;

\begin{enumerate}
\item Create a Hankel matrix
$
{\mathbf H}_{L-N,N+1}= (f^{(k+\ell)}(0))_{k=0,\ell=0}^{L-N-1,N}$ with $f^{(k)}(0)$ given in (\ref{fk0}) 
and compute the SVD ${\mathbf H}_{L-N,N+1}  = {\mathbf U}_{L-N} \, {\mathbf D}_{L-N,N+1} {\mathbf W}_{N+1}$. 
\item Compute the frequencies  $\lambda_1,\ldots ,\lambda_N$ as eigenvalues of the matrix  $\left(  {\mathbf W}_{N}(0)^{T} \right)^{\dagger}  {\mathbf W}_{N}(1)^{T},$
with matrices as in (\ref{matpen3}), where $\left({\mathbf W}_{N}(0)^{T} \right)^{\dagger}$ denotes the Moore-Penrose inverse of  ${\mathbf W}_{N}(0)^{T}$.

\item Compute the coefficients  $\gamma_{j}$, $j=1,\ldots ,N$ as the least-squares solution of the system
$$
\sum\limits_{j=1}^{N} \gamma_j \,  \lambda_j^{k} =f^{(k)}(0), \ k=0,\ldots, L.
$$
\end{enumerate}

\noindent
\textbf{Output:}  parameters $\lambda_j$ and $\gamma_{j}$ for $j=1,\ldots ,N$  such that 
$\sum_{j=1}^{N} \gamma_j \mathrm{e}^{\lambda_j t}$ approximates $\mathrm{e}^{-t^{2}/2\sigma}$.}
\end{algorithm}

\subsection{Comparison with ESPRIT and ESPIRA }
Now we compare our method with  the Prony type methods  ESPRIT and ESPIRA described in \cite{DPP21}.  Note that for Prony type methods  for approximation  of a function $f$ by exponential sums, the  nodes $\mathrm{e}^{\lambda_j}$ are computed as eigenvalues of a Hankel or a Loewner matrix pencil,
where these matrices are constructed  from a finite number of samples of $f$. In a second step,
the coefficients $\gamma_j$ are determined by solving a linear least squares problem with Vandermonde or Cauchy type coefficient matrices. By contrast, our method in Section \ref{appreal}
 is a matrix pencil method, where the involved matrices are constructed from weighted integral values to solve the minimization problem (\ref{minnorm}), and we have shown that the eigenvalues of this matrix pencil are  zeros of a scaled Hermite polynomial.
Such an interpretation of the eigenvalues  is not possible for the matrix pencils appearing for ESPRIT or ESPIRA.  
For a further comparison of Prony-kind methods and the differential approximation method we refer to  \cite{Scha79}.


\begin{figure}[h]
\centering
\begin{subfigure}{0.32\textwidth}
\includegraphics[width=\textwidth]{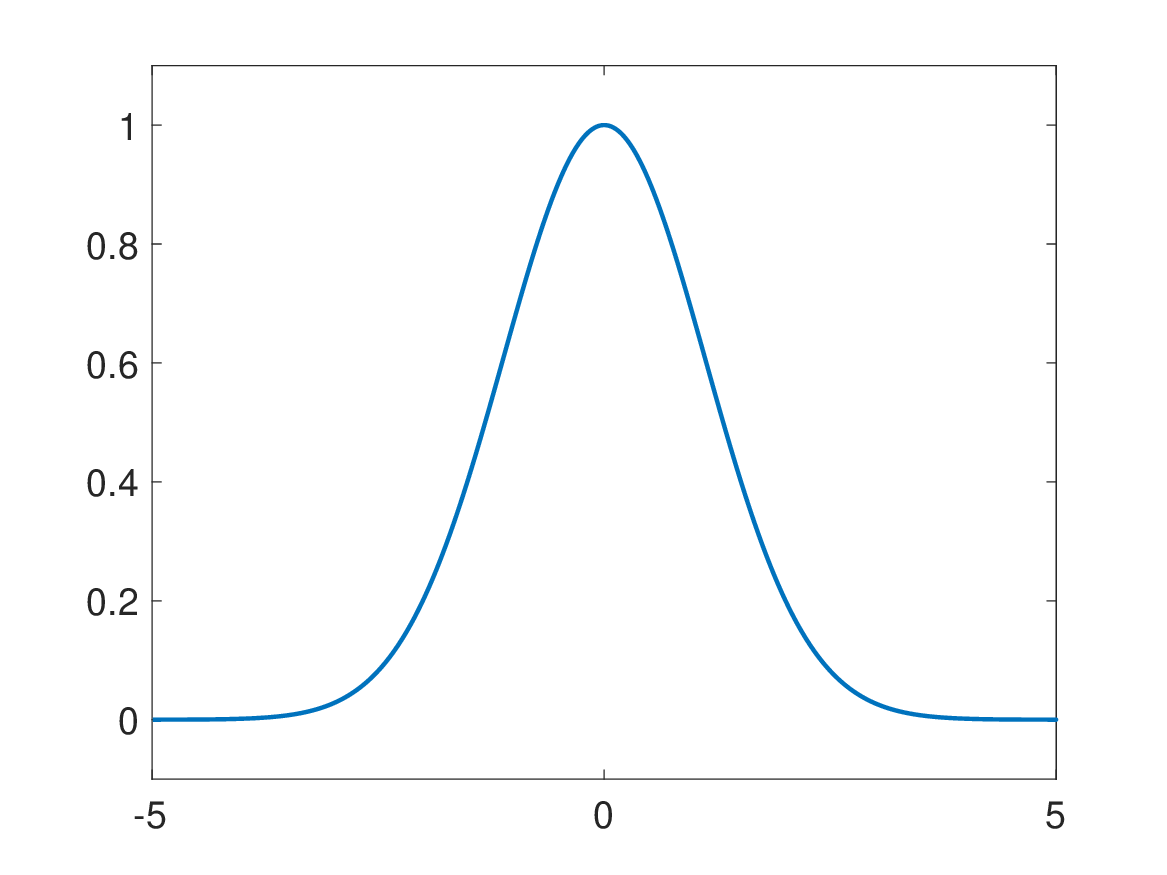}   
\end{subfigure}
\begin{subfigure}{0.32\textwidth}
\includegraphics[width=\textwidth]{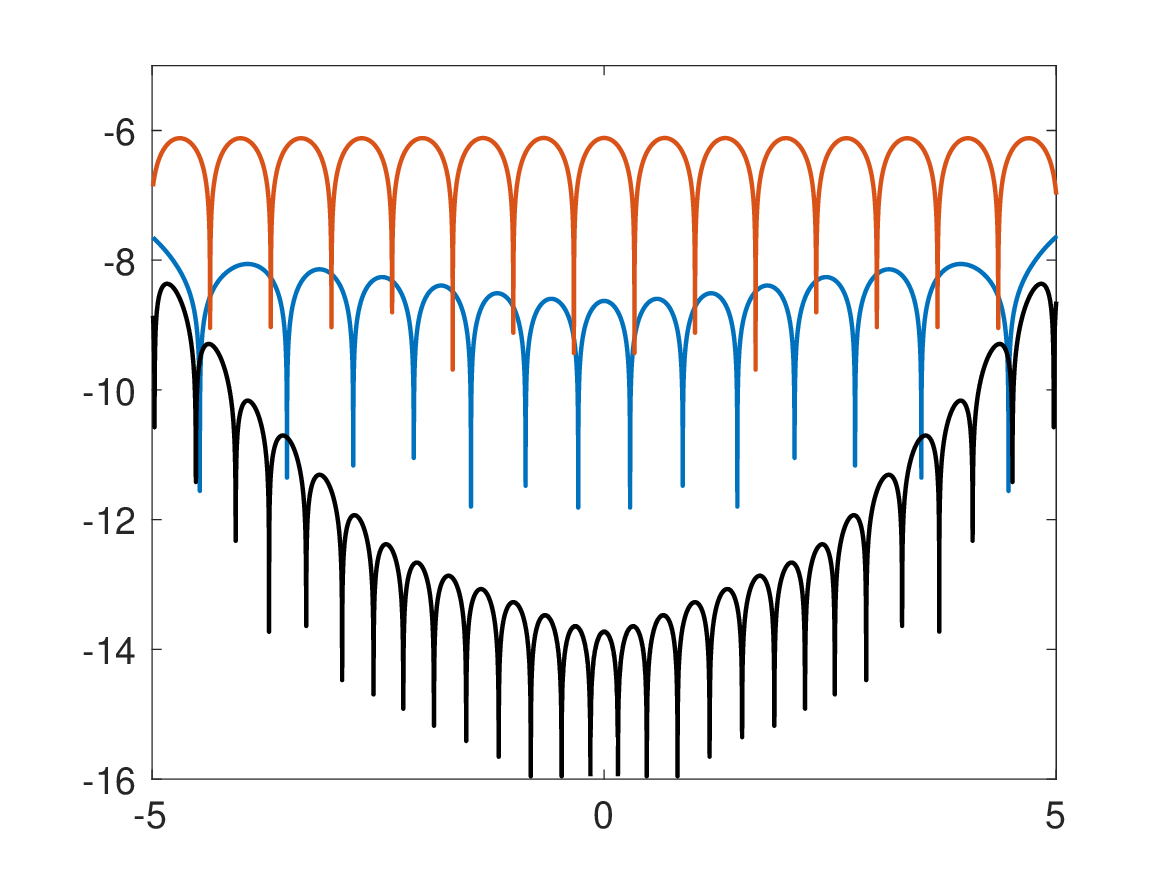}   
\end{subfigure}
\begin{subfigure}{0.32\textwidth}
\includegraphics[width=\textwidth]{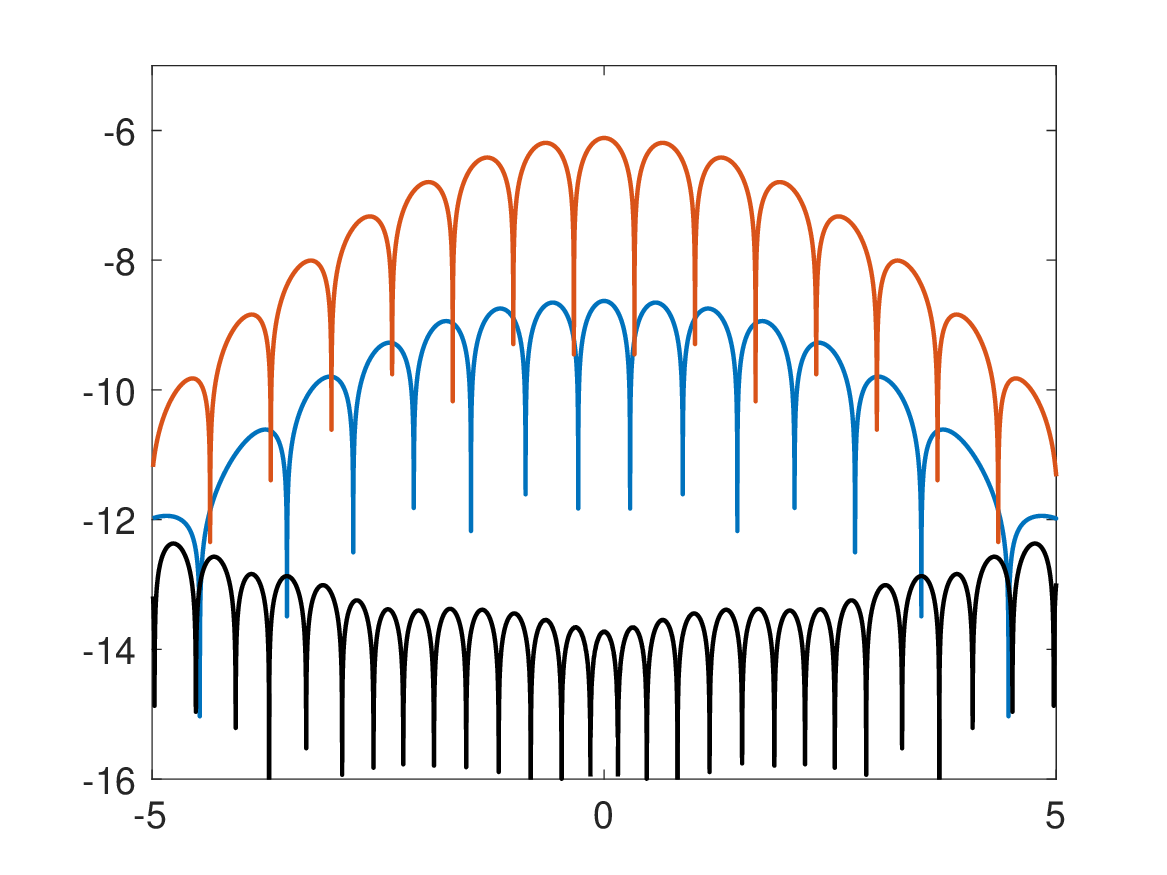}   
\end{subfigure}
\caption{The case $\sigma=1.25$, $\rho=\sigma/2$ and $N=16$ (cosine sum of lenght $8$).   Left: Gaussian function, Middle: Error in $L_\infty([-5,5])$-norm  for ESPRIT (blue), ESPIRA (red) and Algorithm \ref{alg11} (black), Right: error in  $L_\infty([-5,5],\mathrm{e}^{-t^2/4\rho})$ for ESPRIT (blue), ESPIRA (red) and Algorithm \ref{alg11} (black).}
\label{fig33}
\end{figure}
 
Using Prony type methods based on function values we can construct approximations only in a finite segment. 
In or numerical example, we consider the approximation of a function $f(t)=\mathrm{e}^{-t^2/2\sigma}$ by an exponential sum of order $N=16$, i.e., a cosine sum of length $8$, in the interval  $[-2\pi,2\pi]$.  
We take $\sigma=1.25$ and $\rho=\sigma/2$.  For  ESPRIT and ESPIRA we use $L=100$ equidistant sample values at points $t_k=\frac{h(2k+1)}{2}$,  $k=0,\ldots,99$,  with $h=\frac{\pi}{60}$ and employ the fact that $f(t_k)=f(-t_k)$.
We obtain the maximum errors $2.2\cdot 10^{-8}$ for ESPRIT and $7.8\cdot10^{-7}$ for ESPIRA.  Using the differential Prony type method in Algorithm \ref{alg11} we get the error $4.3\cdot 10^{-9}$, see Figure \ref{fig33}. 
In Figure \ref{fig33}, we present the approximation error on a logarithmic scale on the interval $[-5,5]$ for ESPRIT, ESPIRA and Algorithm 1 (with $\rho=\sigma/2$) in the maximum norm and in the weighted $L_\infty([0,5],\mathrm{e}^{-t^2/4\rho})$ norm.

Finally, we note that the direct application of the ESPRIT  or the ESPIRA algorithm for approximation of the Gaussian by an exponential sum $\sum_{j=1}^{N} \gamma_{j} z_{j}^{t}$ (without the restriction that the approximation has to be a sum of cosines with real coefficients) we obtain smaller absolute errors, 
where the performance of the two algorithms depends on the number of given data. ESPIRA outperforms ESPRIT for $L> 350$,  while for smaller $L$, ESPRIT provides the smaller error.
The obtained exponential sums are indeed complex. We have used here the algorithms from \cite{DPP21}.  Note that these methods have a larger complexity than Algorithm \ref{alg11}, $\mathcal{O}(L^{3})$ for ESPRIT and  $\mathcal{O}(L(N^{3}+\log L))$  for ESPIRA.
 

\section{Conclusion}
In this paper, we have applied the differential approximation method to construct a cosine sum that approximates the Gaussian ${\mathrm e}^{-t^{2}/2\sigma}$ with exponential decay.

 Previous approaches, where exponential error convergence rates for completely monotone functions were shown, employed a representation of these functions by the Laplace transform, see \cite{Bbook}, which can then be discretized by a quadrature rule.
In \cite{Greengard2022}, a quadrature rule for the inverse Laplace transform formula (\ref{green}) has been applied. 
 Our convergence proof is conceptionally different and is not based on the Laplace transform.
Instead, it  heavily relies on the Gauss-Hermite quadrature formula and uses the fact that the Gaussian occurs as the weight function for orthogonal Hermite polynomials.

We conclude that  the convergence analysis for the approximation with exponential sums is always closely related to quadrature formulas that converge exponentially for special analytic functions. Consequently, there is also a close relation to
rational approximation, since these quadrature rules are usually related to rational or meromorphic functions \cite{TW14, T21}.
 For further study of the connection between exponential and rational approximation we refer to \cite{BM09, DPP21, DPR23, D24, PP19, WDT21}. 

 It remains an open question, which other smooth functions can be approximated by short exponential sums with the same error convergence rates. For the approximation of the Gaussian,
the key point has been to obtain the frequency parameters $\lambda_j$, $j=1,\ldots,N$ as zeros of a scaled Hermite polynomial.
Thus, the question arises,  whether we can  approximate other weight functions $w$ by exponential sums using the differential  method, thereby obtaining suitable frequency parameters as zeros of scaled orthogonal polynomials.
We may consider orthogonal polynomials $p_n$ with respect to a weight function $w$ in the segment $(a,b)$ (or on ${\mathbb R}$)  satisfying
$$
\textstyle \int\limits_a^{b} p_n(t) p_m(t) w(t) \, \mathrm{d} t = \delta_{nm}. 
$$
These polynomials $p_n$ can be defined by the Rodrigues formula 
\begin{equation}\label{rodgen}
\textstyle p_n(t)=\frac{1}{\alpha_n w(t)} \frac{\mathrm{d}^{n}}{ \mathrm{d} t^{n}} \left(w(t) (q(t))^{n} \right),
\end{equation}
where $\alpha_n$ is some constant and $q$ is some algebraic polynomial of degree at most $2$.
The question is  now the following: 
Can we obtain an approximation of a weight function $w$ (instead of a Gaussian function) by exponential sums $\sum_{j=1}^{n} \gamma_j \mathrm{e}^{\lambda_j t}$ in the segment $(a,b)$ such that the frequencies $\lambda_j$,  $j=1,\ldots,n$, can be expressed  via zeros of these (scaled) orthogonal polynomials $p_n$? 
The case of approximation of Gaussian functions using the described idea is the simplest one, since we have here $q(t)=1$ in (\ref{rodgen}).  
  
\section*{Acknowledgement}
The authors thank Lennart Hübner (KU Leuven) and Yannick Riebe (University Göttingen) for several discussions and improvements of this manuscript.
 The second author acknowledges support by the EU MSCA-RISE-2020 project EXPOWER  and of the DFG CRC 1456.

\section*{Statements and Declarations}
\textbf{Competing Interests:} The authors declare that they have no competing interests.

\small

\end{document}